\documentclass[a4paper,12pt]{amsart}
\usepackage{hyperref}
\usepackage{amsmath,amsthm,amssymb,amscd}
\usepackage{color}

\title[Lipschitz homotopy convergence II]{Lipschitz homotopy convergence of Alexandrov spaces II}

\author[T. Fujioka]{Tadashi Fujioka}
\address[T. Fujioka]{Department of Mathematics, Kyoto University, Kyoto 606-8502, Japan}
\email{fujioka.tadashi.5t@kyoto-u.ac.jp}

\author[A. Mitsuishi]{Ayato Mitsuishi}
\address[A. Mitsuishi]{Department of Applied Mathematics, Fukuoka University, Fukuoka 814-0180, Japan}
\email{mitsuishi@fukuoka-u.ac.jp}

\author[T. Yamaguchi]{Takao Yamaguchi}
\address[T. Yamaguchi]{Department of Mathematics, University of Tsukuba, Tsukuba 305-857, Japan}
\email{takao@math.tsukuba.ac.jp}

\date{\today}

\theoremstyle{plain}
\newtheorem{thm}{Theorem}
\newtheorem{cor}[thm]{Corollary}
\newtheorem{prop}[thm]{Proposition}
\newtheorem{lem}[thm]{Lemma}
\newtheorem{clm}[thm]{Claim}
\newtheorem{subclm}[thm]{Subclaim}
\newtheorem{comp}[thm]{Complement}

\theoremstyle{definition}
\newtheorem{dfn}[thm]{Definition}
\newtheorem{rem}[thm]{Remark}

\newtheorem{conv}[thm]{Convention}

\numberwithin{thm}{section}
\numberwithin{equation}{section}

\DeclareMathOperator{\dist}{dist}

\DeclareMathOperator{\vol}{vol}
\DeclareMathOperator{\supp}{supp}

\newcommand{\step}[1]{\medskip\noindent\textit{#1.}}

\begin{document}

\begin{abstract}
We establish a quantitative version of the Lipschitz homotopy convergence introduced by Mitsuishi and Yamaguchi for a moduli space of compact Alexandrov spaces without collapsing.
Along the way, we obtain a Lipschitz version of Petersen's homotopy stability theorem that is applicable to more general settings, including CAT spaces.
We also show that the Lipschitz homotopies can be chosen to preserve the singular strata of Alexandrov spaces, i.e., extremal subsets.
\end{abstract}

\maketitle

\section{Introduction}

\subsection{Main result}\label{sec:main}

This paper is a continuation of \cite{MY:lip} written by the second and third authors.
The main result of \cite{MY:lip} was that the class of Alexandrov spaces with upper bounds on dimension and diameter and with lower bounds on curvature and volume contains only finitely many Lipschitz homotopy types.
As conjectured in \cite[\S7]{MY:lip}, it is quite natural to expect that there exists a uniform Lipschitz constant for the Lipschitz homotopies.
In this paper, we will solve this problem by establishing a quantitative version of the Lipschitz homotopy stability theorem \cite[1.2]{MY:lip}.

To state the main theorem, we introduce some terminology (see Section \ref{sec:pre} for more details).
Let $X$ and $Y$ be metric spaces.
For a constant $C>0$, we define the {\it $C$-Lipschitz homotopy distance}
\[d_{C\mathchar`-\mathrm{LH}}(X,Y)\]
between $X$ and $Y$ as the infimum of $\epsilon>0$ such that there exist $\epsilon$-approximations $f:X\to Y$ and $g:Y\to X$ that are $C$-Lipschitz, together with $C$-Lipschitz maps $F:X\times [0,\epsilon]\to X$ and $G:Y\times [0,\epsilon]\to Y$ satisfying
\[F(\cdot,0)=g\circ f,\quad F(\cdot,\epsilon)=1_X,\quad G(\cdot,0)=f\circ g,\quad G(\cdot,\epsilon)=1_Y,\]
where $1$ denote the identity maps.
We refer to such $f$ and $g$ as \textit{$(C,\epsilon)$-Lipschitz homotopy approximations} between $X$ and $Y$.
In other words, they are $\epsilon$-approximations that are also $C$-Lipschitz homotopy equivalence whose time interval has length $\epsilon$.

For $n\in\mathbb N$ and $D,v>0$, let
\[\mathcal A(n,D,v)\]
denote the class of compact $n$-dimensional Alexandrov spaces with curvature $\ge-1$, diameter $\le D$, and volume $\ge v$.
Recall that $\mathcal A(n,D,v)$ is compact with respect to the standard Gromov--Hausdorff distance, denoted by $d_\mathrm{GH}$.

The following is the main result of the present paper.

\begin{thm}\label{thm:main}
For given $n$, $D$, and $v$, there exists a constant $C\gg 1$ satisfying the following:
for any $M,M'\in\mathcal A(n,D,v)$ with Gromov--Hausdorff distance less than $C^{-1}$, we have
\[d_{C\mathchar`-\mathrm{LH}}(M,M')<Cd_\mathrm{GH}(M,M').\]
More precisely, for any $\epsilon$-approximation $f:M\to M'$, where $\epsilon<C^{-1}$, there exists a $(C,C\epsilon)$-Lipschitz homotopy approximation $\tilde f:M\to M'$ that is $C\epsilon$-close to $f$.
\end{thm}

As a direct consequence, we have

\begin{cor}\label{cor:main}
$\mathcal A(n,D,v)$ is compact with respect to the $C$-Lipschitz homotopy distance.
There are only finitely many $C$-Lipschitz homotopy types in $\mathcal A(n,D,v)$.
\end{cor}

Our results are new even in the Riemannian case (cf.\ \cite{GP}, \cite{GPW}, \cite{GPW:err}), since the previous results are all topological.
Theorem \ref{thm:main} and Corollary \ref{cor:main} are closely related to Perelman's stability theorem (\cite{Per:alex}, \cite{K:stab}), which asserts that under the same assumption as in Theorem \ref{thm:main}, $M$ and $M'$ are homeomorphic.
In particular, this implies that the number of homeomorphism types in $\mathcal A(n,D,v)$ is finite.
Although Perelman claimed that the stability homeomorphism can be chosen to be bi-Lipschitz, the proof has never been published.
For a recent progress in this direction, see \cite{Al}.

Our approach is actually quite different from the proof of Perelman's stability theorem and applicable to other settings, as we will see in the next subsection.
In particular, the notion of quantitative Lipschitz homotopy convergence is useful beyond the realm of Alexandrov spaces, especially when one cannot expect the homeomorphism stability.
For instance, see the third author's recent joint work \cite{NSY}, \cite{YZ}.

\subsection{Strategies}

There are two ways to prove the main theorem.
One is the same strategy as in the previous paper \cite{MY:lip}, which uses the good coverings of Alexandrov spaces introduced in \cite{MY:good}.
The other strategy is new: a Lipschitz analogue of Petersen's argument in \cite{Peter} (cf.\ \cite{GPW:err}) relying only on the local contractibility of Alexandrov spaces.
The first proof is a refinement of the previous one in \cite{MY:lip} and the second proof is partly a generalization of the first.
Therefore, we will present the two proofs in this order.

In both cases, the key geometric property is the same: the \textit{local uniform Lipschitz contractibility} of Alexandrov spaces, Theorem \ref{thm:conv}.
It asserts that any $\epsilon$-ball in an Alexandrov space $M\in\mathcal A(n,D,v)$ is contained in a convex domain admitting a $C$-Lipschitz contraction of time $\epsilon$, provided $\epsilon<C^{-1}$, where $C=C(n,D,v)$ is a constant depending only on $n$, $D$, and $v$.
This is a refinement of an earlier result of Perelman and Petrunin \cite[4.3]{PP:ext} (cf.\ \cite[7.1.3]{Pet:sc}) on convex hulls in Alexandrov spaces.
The proof is a combination of two techniques: the construction of a strictly concave function introduced by Perelman and Petrunin \cite[4.3]{PP:ext} and the stability of Lipschitz contractible balls proved by the second and third authors \cite[1.2]{MY:stab} (see also Remark \ref{rem:erratum}).
We remark that a similar argument can also be found in \cite{K:reg}.

In the first strategy using good coverings, we need an additional argument based on Theorem \ref{thm:conv} to construct a quantitative version of a good covering (Theorem \ref{thm:good}).
Once the existence of such a good covering is shown, the rest of the proof is almost the same as in the previous paper \cite{MY:lip}.
We prove that any metric space having a quantitative good covering is uniformly Lipschitz homotopy equivalent to a geometric realization of the nerve (Theorem \ref{thm:nerve}).
Via the stability of the nerve of a good covering, we obtain Theorem \ref{thm:main}.

In the second strategy based only on local contractibility, we prove the following Lipschitz version of Petersen's stability theorem \cite[Theorem A]{Peter}.
This is in some sense a generalization of the previous good covering argument.
See Section \ref{sec:main2} for the terminology.

\begin{thm}\label{thm:Peter}
For any $C>0$ there is $\tilde C>0$ satisfying the following:
Let $M$ and $M'$ be metric spaces that are locally $C$-Lipschitz contractible and locally $C$-doubling.
If the Gromov--Hausdorff distance between $M$ and $M'$ is less than $\tilde C^{-1}$, then we have
\[d_{\tilde C\mathchar`-\mathrm{LH}}(M,M')<\tilde Cd_\mathrm{GH}(M,M').\]
More precisely, for any $\epsilon$-approximation $f:M\to M'$, where $\epsilon<\tilde C^{-1}$, there exists a $(\tilde C,\tilde C\epsilon)$-Lipschitz homotopy approximation $\tilde f:M\to M'$ that is $\tilde C\epsilon$-close to $f$.
\end{thm}

Since Alexandrov spaces in $\mathcal A(n,D,v)$ satisfy the above assumptions with some constant $C=C(n,D,v)$, Theorem \ref{thm:Peter} immediately implies Theorem \ref{thm:main}.
The proof of Theorem \ref{thm:Peter} is along the same lines as the original one in \cite{Peter}, with careful checking of the Lipschitz constants.
However, the existence of a Lipschitz domination of an Alexandrov space by a polyhedron requires an additional argument not found in \cite{Peter} (Proposition \ref{prop:dom}).
Fortunately, this is done by modifying an argument in the first approach using good coverings.

\begin{rem}
The above two strategies are applicable not only to Alexandrov spaces, but also to any class of metric spaces having the required coverings (and for which the homeomorphism stability cannot be expected).
For example, CAT spaces with a uniform local doubling constant admit such coverings.
See Remark \ref{rem:CAT} for more details.
\end{rem}

\subsection{Modifications}\label{sec:mod}

There are two modifications of the main theorem.
One is the gluing with an almost isometry as in \cite[1.4]{MY:lip} (asked by John Lott).
For $\epsilon>0$ small, an \textit{$\epsilon$-almost isometry} between metric spaces means a bi-Lipschitz homeomorphism with Lipschitz constants $1\pm\epsilon$.
Recall that there exists an almost isometry between closed subsets of the strained parts of two nearby Alexandrov spaces (Theorem \ref{thm:reg}; see \cite[\S2.3]{MY:lip} for the definition and properties of strainers).

For an $n$-dimensional Alexandrov space $M$, let
\[\mathcal R_M(\delta,\ell)\]
denote the set of points in $M$ with $(n,\delta)$-strainers of length $>\ell$, where $\ell\le 1$ and $\delta$ is less than a sufficiently small constant depending only on $n$.
We also denote by $\tau(\delta)$ a positive function depending only on $n$ such that $\tau(\delta)\to0$ as $\delta\to0$.

\begin{thm}\label{thm:isom}
Under the same assumption as in Theorem \ref{thm:main}, the Lipschitz homotopy approximation $\tilde f:M\to M'$ can be chosen to be a $\tau(\delta)$-almost isometric open embedding on $\mathcal R_M(\delta,\ell)$, provided that the Gromov--Hausdorff distance between $M$ and $M'$ is less than $C^{-1}\delta\ell$.
\end{thm}

Although the proof is along the same lines as in the previous paper \cite{MY:lip}, we have to check the Lipschitz constant more carefully.

The other modification is new: a relative version of the main theorem with respect to extremal subsets.
\textit{Extremal subsets} are singular sets of Alexandrov spaces introduced by Perelman and Petrunin \cite{PP:ext}, which connect the geometric and topological structures of Alexandrov spaces (see Section \ref{sec:ext} for the definition).
It is known that in some sense any Alexandrov space is uniquely stratified into its extremal subsets (\cite{PP:ext}).
Kapovitch \cite[\S9]{K:stab} proved that the stability homeomorphism of Perelman can be chosen to preserve such extremal strata.
It is thus natural to expect that the same holds true for our Lipschitz homotopy equivalence.

As in Subsection \ref{sec:main}, let $f:X\to Y$ and $g:Y\to X$ be $(C,\epsilon)$-Lipschitz homotopy approximations between metric spaces $X$ and $Y$ with $C$-Lipschitz homotopies $F:X\times[0,\epsilon]\to X$ and $G:Y\times[0,\epsilon]\to Y$.
For subsets $A\subset X$ and $B\subset Y$, we say that $f$ and $g$ \textit{respect} $A$ and $B$ if
\[f(A)\subset B,\quad g(B)\subset A,\quad F(A,t)\subset A,\quad G(B,t)\subset B\]
for all $0\le t\le\epsilon$.

\begin{thm} \label{thm:ext}
Under the same assumption as in Theorem \ref{thm:main}, let
\[E\subset M,\quad E'\subset M'\]
be extremal subsets that are $C^{-1}$-close via the Gromov--Hausdorff approximation between $M$ and $M'$.
Then the Lipschitz homotopy approximations between $M$ and $M'$ can be chosen to respect $E$ and $E'$.
\end{thm}

Since a limit of extremal subsets is also extremal (\cite[4.1.3]{Pet:sc}), the above situation naturally occurs under the convergence of Alexandrov spaces (although the limit Alexandrov space may have new extremal subsets).
Note that Theorem \ref{thm:ext} is independent of the previous Theorem \ref{thm:isom} because any proper extremal subset is contained in the non-strained part of the ambient Alexandrov space.

We will give two different proofs of Theorem \ref{thm:ext}.
In the first proof, the key fact is that any extremal subset is invariant under the gradient flows of semiconcave functions.
Since the Lipschitz homotopies of Theorem \ref{thm:main} are defined by gradient flows, it turns out that they automatically preserve extremal subsets.
The second proof is more general, which provides a way to deform any Lipschitz homotopy approximation (not necessarily given by gradient flows) to preserve an extremal subset.
This is done by using a neighborhood deformation retraction to an extremal subset constructed by the first author \cite[1.5]{F:reg}.

\begin{rem}
In Theorem \ref{thm:ext}, one can actually assert that the Lipschitz homotopy approximations between $M$ and $M'$ respect all nearby extremal subsets in $M$ and $M'$ simultaneously (the same is true for Perelman's stability theorem; see \cite[4.3]{HS} for instance).
This is obvious from the first proof explained above, and also can be shown by multiple applications of the second proof.
\end{rem}

\step{Organization}
Section \ref{sec:pre} contains background material.
Section \ref{sec:conv} is devoted to the proof of the key Theorem \ref{thm:conv} on the local Lipschitz contractibility of Alexandrov spaces.
In Section \ref{sec:good}, using it, we construct a quantitative version of a good covering in Theorem \ref{thm:good}.
Section \ref{sec:main1} provides the first proof of Theorem \ref{thm:main} using good coverings.
This section is almost the same as \cite[\S3]{MY:lip}.
Section \ref{sec:main2} provides the second proof of Theorem \ref{thm:main} relying only on local contractibility.
We prove Theorem \ref{thm:Peter} in a way parallel to \cite[\S2--\S4]{Peter}, except for Proposition \ref{prop:dom}.
In Section \ref{sec:isom}, we prove Theorem \ref{thm:isom}.
Finally, in Section \ref{sec:ext}, we give two proofs of Theorem \ref{thm:ext}.

\step{Acknowledgments}
The first author would like to thank Mohammad Alattar for responding to his inquiry regarding \cite{Al}.
The first author was supported by JSPS KAKENHI Grant Number 22KJ2099 (22J00100).
The second author was supported by JSPS KAKENHI Grant Number 20K03598.
The third author was supported by JSPS KAKENHI Grant Number 21H00977.

\section{Preliminaries} \label{sec:pre}

In this section, we fix notation, introduce some terminology for Lipschitz homotopy and the Gromov--Hausdorff convergence, and review the properties of gradient flows in Alexandrov spaces.

In this paper, an \textit{Alexandrov space} means a complete metric space of finite Hausdorff dimension that satisfies the Toponogov triangle comparison.
We refer to \cite{AKP}, \cite{BBI}, \cite{BGP} for the basic concepts on Alexandrov spaces.

\subsection{Notation}

The distance between $x$ and $y$ will be denoted by $|xy|$.
The open $r$-ball and $r$-sphere around a point $p$ are denoted by $B(p,r)$ and $S(p,r)$, respectively.
Similarly, the open $r$-neighborhood of a subset $A$ and its metric boundary are denoted by $B(A,r)$ and $S(A,r)$.
For a metric space $X$, the same space with metric multiplied by $\lambda>0$ is denoted by $\lambda X$.
We denote by $1_X$ the identity of $X$.

The distance function to a point $p$ will be denoted by $|p\cdot|$ or $\dist_p$.
The direction of a shortest path from $p$ to $q$ is denoted by $\uparrow_p^q$ and the set of directions of all shortest paths from $p$ to $q$ is denoted by $\Uparrow_p^q$.
We use the same notation for a subset $A$, i.e., $|A\cdot|$, $\dist_A$, $\uparrow_p^A$, and $\Uparrow_p^A$.
The angle and comparison angle will be denoted by $\angle$ and $\tilde\angle$, respectively.
The space of directions and tangent cone at $p$ will be denoted by $\Sigma_p$ and $T_p$, respectively.

Throughout the paper, $n$, $D$, and $v$ denote, respectively, the dimension, an upper diameter bound, and a lower volume bound of Alexandrov spaces.
Unless otherwise stated, the lower curvature bound is always $-1$.
We use the symbol $C$ to denote various large positive constants, usually depending only on $n$, $D$, and $v$.
We also use the symbol $\epsilon$ to denote a small positive number, usually assumed to be less than some small constant depending only on $n$, $D$, and $v$.
The usage of these symbols will be explained in more detail in Sections \ref{sec:main1} and \ref{sec:main2} (see Conventions \ref{conv:C} and \ref{conv:C'}).

\subsection{Lipschitz homotopy}

We introduce a quantitative version of the Lipschitz homotopy used in \cite{MY:lip}.
Fix $C>0$ and $\epsilon>0$ (usually $C\gg1$ and $\epsilon\ll1$, but not necessarily assumed here).
Loosely speaking, we will add the word ``$(C,\epsilon)$-Lipschitz'' to the terminology of homotopy theory if the homotopy under consideration is $C$-Lipschitz continuous and its time interval has length $\epsilon$.

Let $X$ and $Y$ be metric spaces.
For simplicity, whenever we consider a product space with a time interval, like $X\times[0,\epsilon]$, we equip it with the $L_1$ distance.

\begin{dfn}
Two maps $h_0, h_\epsilon:X\to Y$ are said to be \textit{$(C,\epsilon)$-Lipschitz homotopic} if there exists a $C$-Lipschitz map
\[h:X\times[0,\epsilon]\to Y\]
such that $h_i = h(\cdot,i)$ for $i=0,\epsilon$.
The map $h$ is called a \textit{$(C,\epsilon)$-Lipschitz homotopy} between $h_0$ and $h_\epsilon$.

We say that $X$ and $Y$ are \textit{$(C,\epsilon)$-Lipschitz homotopy equivalent} if there exist $C$-Lipschitz maps
\[f:X\to Y,\quad g:Y\to X\]
such that $g\circ f$ and $f\circ g$ are $(C,\epsilon)$-Lipschitz homotopic to the identities $1_X$ and $1_Y$, respectively.
The maps $f$ and $g$ are called \textit{$(C,\epsilon)$-Lipschitz homotopy equivalences} between $X$ and $Y$.
\end{dfn}

Note that if $h$ is a $(C,C'\epsilon)$-Lipschitz homotopy for some $C'\ge 1$, then
\[h'(\cdot,t):=h(\cdot,C't)\]
is a $(CC',\epsilon)$-Lipschitz homotopy defined for $t\in[0,\epsilon]$.
Hence any large constant multiple of the time (independent of $\epsilon$) can be absorbed into the Lipschitz constant.
This basic fact will be used  implicitly throughout the paper.

The following is a special case of Lipschitz homotopy equivalence.

\begin{dfn}
We say that $A\subset X$ is a \textit{$(C,\epsilon)$-Lipschitz strong deformation retract} of $X$ if there exists a $C$-Lipschitz map
\[F:X\times[0,\epsilon]\to X\]
such that
\[F(x,0)=x,\quad F(x,\epsilon)\in A,\quad F(a,t)=a\]
for any $x \in X$, $a \in A$, and $0\le t\le \epsilon$. 
The map $F$ is called a \textit{$(C,\epsilon)$-Lipschitz strong deformation retraction} of $X$ to $A$.

Moreover, if $A$ is a point, we say that $X$ is \textit{$(C,\epsilon)$-Lipschitz strongly contractible} and $F$ is a \textit{$(C,\epsilon)$-Lipschitz strong contraction}.
In this case $A$ is called the \textit{center} of the contraction.
\end{dfn}

We will often omit the word ``strong'': in fact we will not deal with ``weak'' ones (i.e., not fixing $A$) in this paper.
Note that if $X$ is $(C,\epsilon)$-Lipschitz contractible, then the diameter of $X$ is at most $2C\epsilon$.
This fact will also be used frequently throughout the paper.

We need the following relative version when dealing with extremal subsets.
Let $A\subset X$ and $B\subset Y$ be arbitrary subsets.

\begin{dfn}\label{dfn:rel}
Let $h:X \times [0,\epsilon] \to Y$ be a $(C,\epsilon)$-Lipschitz homotopy.
We say that $h$ \textit{respects} $A$ and $B$ if
\[h(A,t)\subset B\]
for any $0\le t\le\epsilon$.

We say that the pairs $(X,A)$ and $(Y,B)$ are {\it $(C,\epsilon)$-Lipschitz homotopy equivalent} if there exist $(C,\epsilon)$-Lipschitz homotopy equivalences
\[f:X\to Y,\quad g:Y\to X\]
such that $f(A)\subset B$ and $g(B)\subset A$ and the $(C,\epsilon)$-Lipschitz homotopies connecting $g\circ f$ and $f\circ g$ to $1_X$ and $1_Y$ respect $A$ and $B$, respectively.
In this case, $f$ and $g$ are said to \textit{respect} $A$ and $B$.
\end{dfn}

\subsection{Gromov--Hausdorff convergence}

We recall the standard terminology for the Gromov--Hausdorff convergence.
Let $X$ and $Y$ be complete metric spaces.

A map $f:X\to Y$ is called an \textit{$\epsilon$-(Gromov--Hausdorff) approximation} if it satisfies
\begin{enumerate}
\item $||f(x)f(x')| - |xx'||<\epsilon$ for all $x,x'\in X$; 
\item for any $y\in Y$, there is $x\in X$ such that $|f(x)y|<\epsilon$.
\end{enumerate}

If $f:X\to Y$ is an $\epsilon$-approximation, we say that $f(p)\in Y$ is a \textit{lift} of $p\in X$.
Similarly, we say that $\dist_{f(p)}$ is a \textit{lift} of $\dist_p$.
In this way, any object on $X$ defined by distance functions can be lifted to $Y$ by replacing $\dist_p$ with $\dist_{f(p)}$ (e.g., a metric ball).

The \textit{Gromov--Hausdorff distance}
\[d_\mathrm{GH}(X,Y)\]
between $X$ and $Y$ is defined as the infimum of $\epsilon>0$ such that there exist $\epsilon$-approximations $f:X\to Y$ and $g:Y\to X$ satisfying that $g\circ f$ and $f\circ g$ are $\epsilon$-close to the identities $1_X$ and $1_Y$, respectively.
Whenever we discuss the Gromov--Hausdorff distance, we will implicitly fix some approximation $f$ and its ``inverse'' $g$ in the above sense.

Let $A\subset X$ and $B\subset Y$ be closed subsets.
If an $\epsilon$-approximation $f:X\to Y$ satisfies
\[f(A)\subset B(B,\epsilon),\quad B\subset B(f(A),\epsilon),\]
then we call $f$ an \textit{$\epsilon$-approximation} from $(X,A)$ to $(Y,B)$, where $B(\cdot,\cdot)$ denote metric neighborhoods.
This defines the convergence of closed subsets under the Gromov--Hausdorff convergence of the ambient spaces, as used in Theorem \ref{thm:ext}.

If a $(C,\epsilon)$-Lipschitz homotopy equivalence $f:X\to Y$ is also an $\epsilon$-approximation, we call it a \textit{$(C,\epsilon)$-Lipschitz homotopy approximation}.
This leads to the definition of the \textit{$C$-Lipschitz homotopy distance}, as discussed in Subsection \ref{sec:main}.

\subsection{Gradient flows}

We briefly review the properties of the gradient flows of semiconcave functions in Alexandrov spaces.
See \cite{Pet:sc}, \cite{AKP} for more details.
In what follows, for simplicity we will assume that every function is defined on the whole space, but this is not essential.

Let $M$ be an Alexandrov space and $f:M\to\mathbb R$.
In case $M$ has no boundary, $f$ is said to be \textit{$\lambda$-concave}, where $\lambda\in\mathbb R$, if its restriction to every unit-speed shortest path $\gamma(t)$ is $\lambda$-concave, i.e.,
\[f\circ\gamma(t)-(\lambda/2)t^2\]
is concave.
In case $M$ has boundary, we require the same condition for the tautological extension of $f$ to the double of $M$.
A \textit{semiconcave} function is a locally $\lambda$-concave function, where $\lambda$ depends on each point.
A \textit{strictly concave} function is a $\lambda$-concave function with $\lambda<0$.

A typical example of a semiconcave function is a distance function $\dist_A$ restricted to $M\setminus A$, where $A\subset M$ is a closed subset.
Note that the concavity of $\dist_A$ at $x\in M\setminus A$ is bounded above by a constant depending only on a lower bound of $|Ax|$ (and the lower curvature bound $-1$; see \cite[2.4]{MY:loc} for instance).

Let $f:M\to\mathbb R$ be a semiconcave function and $p\in M$.
The directional derivative $df=d_pf$ is a $0$-concave function defined on the tangent cone $T_p$ at $p$.
The \textit{gradient} $\nabla_pf\in T_p$ of $f$ at $p$ is an element characterized by the following two properties:
\begin{enumerate}
\item $df(v)\le\langle\nabla_pf,v\rangle$ for any $v\in T_p$;
\item $df(\nabla_pf)=|\nabla_pf|^2$.
\end{enumerate}
Here $|\cdot|$ and $\langle\cdot,\cdot\rangle$ denote the norm and the scalar product on the tangent cone, respectively, defined in the same way as Euclidean plane (see \cite[\S1.3]{Pet:sc}).
The concavity of $df$ guarantees the existence and uniqueness of the gradient.

The \textit{gradient curve} of a semiconcave function $f$ is a curve whose right tangent vector (\cite[\S2.1]{Pet:sc}) is unique and equal to the gradient of $f$.
For any $\lambda$-concave function $f$, the gradient curve $\Phi(x,t)$ starting at $x\in M$ exists and is unique for all $t\ge0$. 
The map
\[\Phi:M\times[0,\infty)\to M\]
is called the \textit{gradient flow} of $f$.

The most important property of the gradient flow for this paper is its Lipschitz continuity.
Recall that we are considering the $L_1$ distance on a product space with a time interval.

\begin{prop}[{\cite[2.1.4]{Pet:sc}}]\label{prop:lip}
Let $f$ be a $\lambda$-concave function on an Alexandrov space $M$ and $\Phi:M\times[0,\infty)\to M$ its gradient flow.
Then
\[\Phi_t:=\Phi(\cdot,t):M\to M\]
is $e^{\lambda t}$-Lipschitz for any $t\ge 0$.
In particular, if $f$ is $C$-Lipschitz, the restriction of $\Phi$ to $M\times[0,T]$ is $\max\{C,e^{\lambda T},1\}$-Lipschitz.
\end{prop}

We conclude this section with two basic lemmas regarding the gradient, which will be used frequently later.
The first is a gradient estimate for a strictly concave function.

\begin{lem}\label{lem:grad}
Let $f$ be a $(-\lambda)$-concave function on an Alexandrov space $M$, where $\lambda>0$.
If $p\in M$ is the (unique) maximum point of $f$, then for any $x\in M\setminus\{p\}$, we have
\[|\nabla_xf|\ge df(\uparrow_x^p)\ge\lambda|px|/2.\]
\end{lem}

\begin{proof}
The first inequality follows from the property (1) of the gradient and the second follows follows from the $(-\lambda)$-concavity,
\[df(\uparrow_x^p)\ge\frac{f(p)-f(x)+\lambda|px|^2/2}{|px|},\]
where $f(p)\ge f(x)$.
\end{proof}

The second is about the following regularity introduced in \cite{MY:stab}.
For $c>0$ and a subset $A\subset M$, we say that a function $f$ on $M$ is \textit{$c$-regular (outside $A$) in the direction to $A$} if
\[df(\uparrow_x^A)>c\]
for any $x\in M\setminus A$ and some shortest path from $x$ to $A$.

\begin{lem}[{\cite[2.3]{MY:stab}}]\label{lem:reg}
Let $f$ be a semiconcave function on an Alexandrov space $M$ that is $c$-regular in the direction to $A\subset M$.
Then the gradient flow of $f$ outside $A$ monotonically decreases the distance to $A$ with velocity at least $c$.
\end{lem}

\begin{proof}
This follows from the first variation formula and the property (1) of the gradient.
Indeed, if $\alpha(t)$ is a gradient curve of $f$ contained in $M\setminus A$, we have
\[\frac d{dt}|A\alpha(t)|=-\langle\nabla_{\alpha(t)}f,\Uparrow_{\alpha(t)}^A\rangle\le-df(\uparrow_{\alpha(t)}^A)<-c,\]
as required.
\end{proof}

\section{Lipschitz contractible convex hulls}\label{sec:conv}

In this section we prove the following key theorem.
Fixing $n$, $D$, and $v$, we denote by $\mathcal A(n,D,v)$ the class of Alexandrov spaces with dimension $n$, curvature $\ge-1$, diameter $\le D$, and volume $\ge v$.

\begin{thm}[cf.\ {\cite[4.3]{PP:ext}}]\label{thm:conv}
There is $C=C(n,D,v)>0$ satisfying the following:
for any $M \in\mathcal A(n,D,v)$, $p \in M$, and $0<\epsilon<C^{-1}$, there exists a $(C,\epsilon)$-Lipschitz strongly contractible convex domain $U$ containing $B(p,\epsilon)$.

More precisely, the following hold:
\begin{enumerate}
\item there exists a $1$-Lipschitz $(-(C\epsilon)^{-1})$-concave function $h$ defined on a neighborhood of $p$ such that for any $0\le t\le C^{-1}$,
\[U(t):=\{h>-t\epsilon\}\]
is the desired convex domain $U$;
\item $U(t)$ also contains the $(C^{-1}\epsilon)$-ball around its peak $q$, i.e., the unique maximum point of $h$;
\item for any $M'\in\mathcal A(n,D,v)$ that is $(C^{-1}\epsilon)$-close to $M$ with respect to the Gromov--Hausdorff distance, there is a lift $h'$ of $h$ to $M'$ that satisfies the same properties as (1) and (2) with $p$ replaced by its lift $p'\in M'$.
\end{enumerate}
\end{thm}

\begin{rem}
The center of the Lipschitz contraction of $U$ (i.e., the point to which $U$ contracts) is not necessarily $p$ or $q$.
\end{rem}

\begin{rem}\label{rem:diam}
The $(C,\epsilon)$-Lipschitz contractibility of $U$ implies that the diameter of $U$ is at most $2C\epsilon$.
The existence of such a convex hull of uniformly bounded diameter was proved by Perelman and Petrunin \cite[4.3]{PP:ext} (cf.\ \cite[7.1.3]{Pet:sc}).
The above theorem is its refinement.
\end{rem}

The detailed statements (1)--(3) in the second half will be used in the next section to construct a quantitative version of a good covering in Theorem \ref{thm:good}, which is the key ingredient in the first proof of Theorem \ref{thm:main} in Section \ref{sec:main1}.
For the second proof in Section \ref{sec:main2}, the simpler statement in the first half is sufficient.

The proof of Theorem \ref{thm:conv} is a combination of the two techniques from \cite[4.3]{PP:ext} and \cite[1.2]{MY:stab} (see also \cite{K:reg}).
The construction of the strictly concave function $h$ is due to Perelman and Petrunin \cite[4.3]{PP:ext}.
We use the gradient flow of $h$ to contract $U$ to a tiny ball, while we need the stability of contractible balls as in \cite[1.2]{MY:stab} to crush this tiny ball.
Thus the first half of the following proof is the same as \cite[4.3]{PP:ext} and the second half is the same as \cite[1.2]{MY:stab}.
For the convenience of the reader, we include the complete proof.

\begin{rem}\label{rem:erratum}
There is an error in the proof of 
\cite[1.2]{MY:stab}. However,
the statement of \cite[1.2]{MY:stab} holds true 
as stated there (see the erratum in \cite{MY:stab}).
It should be emphasized that we do not directly use \cite[1.2]{MY:stab} in the present paper. 
\end{rem}

\begin{proof}[Proof of Theorem \ref{thm:conv}]
We prove it by contradiction.
Suppose there exist $C_i\to\infty$, $0<\epsilon_i<C_i^{-1}$, $M_i\in\mathcal A(n,D,v)$, and $p_i\in M_i$ such that the statement does not hold, i.e., there is no $(C_i,\epsilon_i)$-Lipschitz strongly contractible convex domain containing $B(p_i,\epsilon_i)$ with the desired properties.
Instead of $M_i$, we will consider $\tilde M_i:=\epsilon_i^{-1}M_i$, and will write metric balls and metric spheres in $\tilde M_i$ as $\tilde B(\cdot,\cdot)$ and $\tilde S(\cdot,\cdot)$, respectively.

For some subsequence (denoted by the same $i$) and some constant $C>0$ independent of $i$, we will construct a function $\tilde h_i$ defined on a large neighborhood of $p_i$ in $\tilde M_i$ with the following properties:
\begin{enumerate}
\item[($\tilde 1$)] $\tilde h_i$ is a $1$-Lipschitz $(-C^{-1})$-concave function such that for any $0\le t\le C^{-1}$, the superlevel set $\tilde U_i(t):=\{\tilde h_i>-t\}$ is a $(C,1)$-Lipschitz contractible domain containing $\tilde B(p_i,1)$;
\item[($\tilde 2$)] $\tilde U_i(t)$ also contains the $C^{-1}$-ball around its peak $q_i$;
\item[($\tilde 3$)] for any $M_i'\in\mathcal A(n,D,v)$ such that $\tilde M_i':=\epsilon_i^{-1}M_i'$ is $C_i^{-1}$-close to $\tilde M_i$, there is a lift $\tilde h_i'$ of $\tilde h_i$ to $\tilde M_i'$ that satisfies the same properties as ($\tilde 1$) and ($\tilde 2$) with $p_i$ replaced by its lift $p_i'\in\tilde M_i'$.
\end{enumerate}
Clearly this contradicts the assumption at the beginning of the proof (consider $h_i:=\epsilon_i\tilde h_i$ on $M_i$).

Since $\epsilon_i\to 0$, we may assume that $(\tilde M_i,p_i)$ converges to a nonnegatively curved Alexandrov space $(N,p)$ of dimension $n$ in the pointed Gromov--Hausdorff topology.
Moreover, the existence of a uniform lower volume bound $v$ for $M_i$ implies that the ideal boundary $N(\infty)$ of $N$ has dimension $n-1$ (see \cite{Shi} for the ideal boundary).
Indeed, by the Bishop--Gromov inequality (\cite[\S6]{Y:conv}), we have
\[\frac{\vol_n(\tilde B(p_i,R))}{R^n}=\frac{\vol_n (B(p_i,\epsilon_iR))}{\epsilon_i^nR^n}\ge\frac{\vol_n(B(p_i,D))}{D^n}\ge\frac v{D^n}\]
for any $i\gg R\gg 1$, where $\vol_n$ denotes the $n$-dimensional Hausdorff measure.
Taking $i\to\infty$, this implies that the volume ratio $\vol_n(B(p,R))/R^n$ of $N$ is uniformly bounded below independent of $R\gg 1$ (\cite[10.8]{BGP}).
Taking $R\to\infty$, we see that the asymptotic cone of $N$ has dimension $n$, and thus $N(\infty)$ has dimension $n-1$ (\cite[1.1]{Shi}).

In what follows, we first prove the statements for the limit space $N$ that correspond to the conditions ($\tilde 1$) and ($\tilde 2$) and then show that the construction can be lifted to $\tilde M_i$.
In particular, the condition ($\tilde 3$) is obvious since $\tilde M_i'$ also converges to the same $N$.

\step{Step 1}
This part is the same as the proof of \cite[4.3]{PP:ext}.
We construct a strictly concave function $h$ defined on an arbitrarily large neighborhood of $p\in N$, which can be lifted to a function $\tilde h_i$ on $\tilde M_i$ with the same concavity.
We show that a superlevel set of $h$ is a convex domain containing the unit ball around $p$ and the same holds for $\tilde h_i$.

Take a maximal $\delta$-discrete set $\{l_\alpha\}_{\alpha\in A}$ in $N(\infty)$, where $\delta>0$ will be determined in Claim \ref{clm:conc}.
For each $l_\alpha$ there exists a corresponding ray $r_\alpha$ emanating from $p$, which yields the Busemann function $b_\alpha$, i.e.,
\[b_\alpha:=\lim_{t\to\infty}|r_\alpha(t)\cdot|-t.\]

Choose $R\gg1$ large enough such that $R^{-1}S(p,R)$ is sufficiently close to $N(\infty)$ in the Gromov--Hausdorff distance, so that the following holds: for any $x\in S(p,R)$, there exists $A_x\subset A$ such that
\begin{equation}\label{eq:a}
\# A_x\ge c/\delta^{n-1}
\end{equation}
for some $c>0$ independent of $x$, and
\begin{equation}\label{eq:b}
b_\alpha(p)\ge b_\alpha(x)+R/2
\end{equation}
for any $\alpha\in A_x$.
Indeed, to get the second inequality \eqref{eq:b}, it suffices to choose $A_x$ so that the subset $\{l_\alpha\}_{\alpha\in A_x}\subset N(\infty)$ is contained in the $\pi/4$-neighborhood of the image of $x$ under the Gromov--Hausdorff approximation $R^{-1}S(p,R)\approx N(\infty)$.
The first inequality \eqref{eq:a} then follows from the fact that $\dim N(\infty)=n-1$ and the Bishop--Gromov inequality.
In particular, $c$ depends only on $n$, $D$, and $v$.

Suppose $\hat R\gg R$, which will be determined by the next sentence and the proof of Claim \ref{clm:conc}.
Setting $q_\alpha:=r_\alpha(\hat R)$ and choosing $\hat R$ large enough, we may assume that for any $\alpha\in A_x$, the following inequality corresponding to \eqref{eq:b} holds:
\[|q_\alpha p|\ge|q_\alpha x|+R/3.\]
We then define a function $h_x$ on $B(p,2R)$ by
\[h_x:=\frac1{\# A_x}\sum_{\alpha\in A_x}\phi(|q_\alpha\cdot|),\]
where $\phi$ is a $C^2$ function such that $1/2\le\phi'\le 1$ and $\phi''=-(8R)^{-1}$ on $[\hat R-2R,\hat R+2R]$.

Clearly $h_x$ is $1$-Lipschitz, $h_x(p)\ge h_x(x)+R/6$, and $h_x(p)=h_{x'}(p)$ for any $x\neq x'\in S(p,R)$.
Moreover,

\begin{clm}[{\cite[3.6]{Per:mor}}]\label{clm:conc}
$h_x$ is $(-\lambda)$-concave on $B(p,2R)$ for some $\lambda>0$ depending only on $c$ and $R$, provided $\delta$ is small enough and $\hat R$ is large enough.
\end{clm}

\begin{proof}
Let $\phi$ be as above.
If $f$ is a $C^2$ function on $\mathbb R$, we have
\[(\phi\circ f)''=\phi''\circ f\cdot(f')^2+\phi'\circ f\cdot f''.\]
Hence
\begin{enumerate}
\item[(i)] if $|f(0)-\hat R|\le2R$, then $(\phi\circ f)''(0)\le f''(0)$;
\item [(ii)] if in addition $|f'(0)|\ge a>0$ and $f''(0)\le a^2/16R$, then $(\phi\circ f)''(0)\le-a^2/16R$.
\end{enumerate}
These estimates can be justified even if $f$ is a semiconcave function that is not $C^2$, by regarding $f''\le\lambda$ as meaning that $f$ is $\lambda$-concave.

Let $c'\ll c$ be a small positive constant that will be determined later, where $c$ is a constant from \eqref{eq:a}.
Let $\gamma(t)$ be an arbitrary unit-speed shortest path with $\gamma(0)\in B(p,2R)$.
We denote by $A_x'\subset A_x$ the set of $\alpha\in A_x$ such that
\[||\gamma'(0)\Uparrow_{\gamma(0)}^{q_\alpha}|-\pi/2|\le c'.\]
Choosing $\hat R\gg R$, we may assume that $\{\uparrow_{\gamma(0)}^{q_\alpha}\}_{\alpha\in A}$ is $\delta/2$-discrete.
Then a standard comparison argument shows
\begin{equation}\label{eq:c}
\#A_x'\le C_nc'/\delta^{n-1},
\end{equation}
provided $\delta\ll c'$, where $C_n$ is a constant depending only on $n$.

Choose $R'\gg R$ large enough so that $|q_\alpha\cdot|$ is $(\sin^2c'/16R)$-concave on $B(p,2R)$ (recall that $N$ has nonnegative curvature).
Then for any $\alpha\in A_x$, the estimate (i) implies
\[\phi(|q_\alpha\gamma|)''(0)\le\sin^2 c'/16R.\]
Furthermore, for any $\alpha\in A_x\setminus A_x'$, the estimate (ii) implies
\[\phi(|q_\alpha\gamma|)''(0)\le-\sin^2 c'/16R.\]
By \eqref{eq:a} and \eqref{eq:c}, assuming $c'\ll c$, we have $\#A_x'\le\#A_x/10$.
Taking the average of $\phi(|q_\alpha\gamma|)''(0)$, we obtain $(h_x\circ\gamma)''(0)\le-\sin^2 c'/100R^2$, as desired.
\end{proof}

Let $h$ be the infimum of $h_x$ over all $x\in S(p,R)$, which is also $1$-Lipschitz and $(-\lambda)$-concave.
Furthermore, $h(p)\ge\sup_{x\in S(p,R)}h(x)+R/6$.
Set
\[U(t):=h^{-1}(a+t,\infty)\cap B(p,R),\quad a:=\sup_{x\in S(p,R)}h(x)\]
for each $0\le t\le R/10$.
Then $U(t)$ is convex.
Moreover, since $h$ is $1$-Lipschitz and $R\gg 1$, it follows that $U(t)$ contains $B(p,1)$.
Note that $U(t)$ also contains a small ball around its peak $q$, i.e., the unique maximum point of $h$.

Finally, we lift the situation to $\tilde M_i$.
Let $\tilde h_i$ and $\tilde U_i(t)$ denote the lifts of $h$ and $U(t)$, respectively, defined by the same formulas with $q_\alpha$ replaced by its lift.
Then $\tilde h_i$ is $1$-Lipschitz and $(-\lambda)$-concave on $\tilde B(p_i,2R)$: the latter follows from the same argument using volume comparison as in Claim \ref{clm:conc} (note $\dim\tilde M_i=n$).
Therefore $\tilde U_i(t)$ is a convex domain containing $\tilde B(p_i,1)$.
This shows the condition ($\tilde 1$) except for the Lipschitz contractibility.
Since the unique maximum point $q_i$ of $\tilde h_i$ converges to $q$, the condition ($\tilde 2$) is clear.

From now on we fix $t$ and denote $U:=U(t)$ and $\tilde U_i:=\tilde U_i(t)$.
We are going to prove the uniform Lipschitz contractibility of $\tilde U_i$ as claimed in the condition ($\tilde 1$).

\step{Step 2}
Here we first show that $U$ can be contracted to a tiny ball $B$ centered at $q$ by the gradient flow of $h$ and this $B$ can be crushed to $q$ by another gradient flow.
We then check that $\tilde U_i$ can be contracted into the corresponding tiny ball $\tilde B_i$ centered at $q_i$ in the same manner.
However, crushing $\tilde B_i$ requires a more careful argument, which will be deferred to Step 3.

As above, $q\in U$ denotes the unique maximum point of $h$.
For any $x\in U\setminus\{q\}$, by Lemma \ref{lem:grad} we have
\[dh(\uparrow_x^q)\ge\lambda|qx|/2.\]
In particular, for any fixed $r>0$, the function $h$ is $(\lambda r/4)$-regular outside $B(q,r/2)$ in the direction to $q$ (see the paragraph before Lemma \ref{lem:reg} for the terminology).
Hence Lemma \ref{lem:reg} implies that the gradient flow of $h$ pushes $U$ into $B(q,r/2)$ within time at most $8R/\lambda r$.
Since $h$ is $1$-Lipschitz and $0$-concave, this flow is $1$-Lipschitz (Proposition \ref{prop:lip}).
In conclusion,

\begin{clm}\label{clm:flow1}
The gradient flow $\Phi$ of $h$ gives a $(1,8R/\lambda r)$-Lipschitz homotopy pushing $U$ into $B(q,r/2)$.
\end{clm}

We choose $r>0$ so small that ${(2r)}^{-1}S(q,2r)$ is sufficiently close to the space of directions $\Sigma_q$ at $q$ in the Gromov--Hausdorff distance (the choice of $r$ will be determined by the next sentence and the argument in Step 3).
In particular, there exists $\nu>0$ such that for any $x\in B(q,r)\setminus\{q\}$ there is $y\in S(q,2r)$ with
\[\tilde\angle qxy>\pi/2+\nu\]
(here $\nu\to\pi/2$ as $r\to 0$, but we cannot use such a strong regularity once the situation is lifted to $\tilde M_i$; see Claim \ref{clm:stab}).
This implies that the gradient flow of $\dist_{S(q,2r)}$ gives a Lipschitz strong contraction of $B(q,r)$ to $q$.
Indeed, for any $x\in B(q,r)\setminus\{q\}$, let $y\in S(q,2r)$ be as above and let $z\in S(q,2r)$ be closest to $x$.
Then we have
\[d\dist_{S(q,2r)}(\uparrow_x^q)\ge-\cos\tilde\angle zxq\ge-\cos\tilde\angle yxq\ge\sin\nu.\]
In other words, $\dist_{S(q,2r)}$ is $(\sin\nu)$-regular in the direction to $q$ on $B(q,r)\setminus\{q\}$.
Hence Lemma \ref{lem:reg} implies that the gradient flow of $\dist_{S(q,2r)}$ crushes $B(q,r)$ to $q$ within time at most $r/\sin\nu$.
This flow is $c(r,\nu)$-Lipschitz on $B(q,r)$, where $c(r,\nu)$ is a constant depending only on $r$ and $\nu$ (Proposition \ref{prop:lip}; actually independent of $r$).
In conclusion,

\begin{clm}\label{clm:flow2}
The gradient flow $\Psi$ of $\dist_{S(q,2r)}$ gives a $(c(r,\nu),r/\sin\nu)$-Lipschitz contraction of $B(q,r)$ to $q$.
\end{clm}

Combining the two gradient flows in Claims \ref{clm:flow1} and \ref{clm:flow2}, we get a $(C,1)$-Lipschitz strong contraction of $U$ to $q$ for some $C>0$.

We now verify that the same construction works for $\tilde U_i$.
As before, let $q_i\in\tilde U_i$ denote the unique maximum point of $\tilde h_i$, which converges to $q$.
The same argument as Claim \ref{clm:flow1} using the $(-\lambda)$-concavity of $\tilde h_i$ shows that $\tilde U_i$ can be pushed into $\tilde B(q_i,r/2)$ by the gradient flow $\Phi_i$ of $\tilde h_i$ with uniform Lipschitz constant and time.
However, in general $\tilde B(q_i,r/2)$ is not Lipschitz strongly contractible to $q_i$.
To contract this ball, we replace $q_i$ with another nearby point, which is the final step.

\step{Step 3}
This part is the same as the proof of \cite[3.5]{MY:stab}, a noncollapsing version of \cite[3.2]{Y:ess}.
We show that there exists $\hat q_i\in\tilde M_i$ converging to $q$ such that the distance function from $\tilde S(q_i,2r)$ is uniformly regular (independent of $i$) on $\tilde B(\hat q_i,r)\setminus\{\hat q_i\}$ in the direction to $\hat q_i$.
Then the same argument as Claim \ref{clm:flow2} shows that $\tilde B(\hat q_i,r)$ is contractible to $\hat q_i$ with uniform Lipschitz constant and time.

The following construction is similar to Step 1, but here we use the space of directions $\Sigma_q$ instead of the ideal boundary $N(\infty)$.
Recall that $(2r)^{-1}S(q,2r)$ is Gromov--Hausdorff close to $\Sigma_q$.
For fixed $0<\theta<\pi/100$, we take a maximal $\theta r$-discrete set $\{x_\alpha\}_\alpha$ in $S(q,2r)$.
For each $\alpha$ we take a maximal $\delta r$-discrete set $\{x_{\alpha\beta}\}_{\beta=1}^{N_\alpha}$ in $S(q,2r)\cap B(x_\alpha,2\theta r)$, where $0<\delta\ll\theta$ will be determined in Claim \ref{clm:stab}.
Then the Bishop--Gromov inequality on $\Sigma_q$ implies
\begin{equation}\label{eq:low}
N_\alpha\ge c\left(\theta/\delta\right)^{n-1},
\end{equation}
where $c>0$ depends only on $n$, $D$, and $v$.

We define a function $f$ on $B(q,r)$ by
\[f:=\min_\alpha f_\alpha,\quad f_\alpha:=\frac1{N_\alpha}\sum_{\beta=1}^{N_\alpha}|x_{\alpha\beta}\cdot|.\]
Since $\{x_\alpha\}_\alpha$ is a $\theta r$-net and $\theta<\pi/100$, it is easy to see that
\[f(q)\ge f(x)+|qx|/2\]
for any $x\in B(q,r)$ (see for instance \cite[3.3]{Y:ess}).

Now we lift the situation to $\tilde M_i$.
Let $f_i$ denote the lift of $f$.
We define $\hat q_i$ as a maximum point of $f_i$ on $\tilde B(q_i,r)$, which converges to the unique maximum point $q$ of $f$.

\begin{clm}\label{clm:stab}
There exists $\nu>0$ such that for any sufficiently large $i$ and any $x_i\in\tilde B(\hat q_i,r)\setminus\{\hat q_i\}$, there is $y_i\in\tilde S(q_i,2r)$ with
\[\tilde\angle\hat q_ix_iy_i>\pi/2+\nu.\]
\end{clm}

\begin{proof}
Suppose the claim does not hold.
Passing to a subsequence, we can assume that there are $\nu_i\to 0$ and $x_i\in\tilde B(\hat q_i,r)\setminus\{\hat q_i\}$ such that
\begin{equation}\label{eq:xi}
\tilde\angle\hat q_ix_iy_i\le\pi/2+\nu_i
\end{equation}
for any $y_i\in\tilde S(q_i,2r)$.
Note that $|\hat q_ix_i|\to 0$ since the statement corresponding to Claim \ref{clm:stab} holds on $N$ due to the choice of $r$ (see Step 2).
We may assume that $(|\hat q_ix_i|^{-1}\tilde M_i,\hat q_i)$ converges to a nonnegatively curved Alexandrov space $(\hat N,\hat q)$.

Let $x_{\alpha\beta}^i\in\tilde S(q_i,2r)$ and $f_\alpha^i$ denote the lifts of $x_{\alpha\beta}$ and $f_\alpha^i$, respectively, that define $f_i$.
Let $v_i,v_{\alpha\beta}^i\in\Sigma_{\hat q_i}$ denote the directions of shortest paths from $\hat q_i$ to $x_i$ and $x_{\alpha\beta}^i$, respectively (choose them so that the following first variation formula holds).
Fix $\alpha$ such that $df_i(v_i)=df_\alpha^i(v_i)$ for infinitely many $i$.
Since $f_i$ attains a maximum at $\hat q_i$, this implies
\[\frac1{N_\alpha}\sum_\beta-\cos\angle(v_i,v_{\alpha\beta}^i)\le 0.\]
We may assume that $v_i$ and $v_{\alpha\beta}^i$ converge to the directions $v,v_{\alpha\beta}\in\Sigma_{\hat q}$ of a shortest path and rays emanating from $q$, respectively.
Then the lower semicontinuity of angle implies
\begin{equation}\label{eq:up}
\frac1{N_\alpha}\sum_\beta-\cos\angle(v,v_{\alpha\beta})\le 0.
\end{equation}

The assumption \eqref{eq:xi} at the beginning of the proof implies
\[\tilde\angle x_i\hat q_ix_{\alpha\beta}^i\ge\pi/2-\nu_i-\mu_i,\]
where $\mu_i\to 0$.
Passing to the limit and using the monotonicity of angle, we obtain $\angle(v,v_{\alpha\beta})\ge\pi/2$.
Moreover, since $\{x_{\alpha\beta}\}_{\beta}$ is $\delta r$-discrete, we see that $\{v_{\alpha\beta}\}_\beta$ is $\delta/2$-discrete.

The rest of the proof is similar to the proof of Claim \ref{clm:conc}.
We denote by $N_\alpha'$ the number of $1\le\beta\le N_\alpha$ such that $\angle(v,v_{\alpha\beta})=\pi/2$.
Then we have
\[N_\alpha'\le C_n/\delta^{n-2},\]
where $C_n$ is a constant depending only on $n$.
Since $\angle (v,v_{\alpha\beta})\ge\pi/2$, this together with \eqref{eq:low} implies that for the majority of $\beta$, we have $\angle (v,v_{\alpha\beta})>\pi/2$, provided $\delta\ll c$.
Hence the average of $-\cos\angle(v,v_{\alpha\beta})$ should be positive, which contradicts \eqref{eq:up}.
\end{proof}

As we have seen in Claim \ref{clm:flow2}, Claim \ref{clm:stab} is enough to conclude that the gradient flow $\Psi_i$ of $\dist_{\tilde S(q_i,2r)}$ is a Lipschitz strong contraction of $\tilde B(\hat q_i,r)$ to $\hat q_i$ with uniform Lipschitz constant and time.
Thus its composition with the gradient flow $\Phi_i$ of $\tilde h_i$ gives a Lipschitz weak contraction of $\tilde U_i$ to $\hat q_i$ (the weakness means that it does not necessarily fix $\hat q_i$).
Furthermore, it is easy to modify $\Phi_i$ to fix a small neighborhood of $q_i$.
For example, the following flow fixes $\tilde B(q_i,r/10)$ while keeping the Lipschitz constant and time uniformly bounded above:
\[\hat \Phi_i(x,t):=\Phi_i(x,\rho(|q_ix|)t),\]
where $\rho:[0,\infty)\to[0,1]$ is a $5/r$-Lipschitz function such that $\rho\equiv0$ on $[0,r/10]$ and $\rho\equiv1$ on $[r/5,\infty)$.
Since $\hat q_i$ converges to $q$ as well as $q_i$, this modified flow fixes $\hat q_i$ for sufficiently large $i$.
Therefore, the composition of $\hat\Phi_i$ and $\Psi_i$ gives a $(C,1)$-Lipschitz strong contraction of $\tilde U_i$ to $\hat q_i$, where $C>0$ is a constant independent of $i$.
This completes the proof of Theorem \ref{thm:conv}.
\end{proof}

\begin{rem}
It is further possible to show that the distance to the center $\hat q_i\in\tilde U_i$ is nonincreasing during the contraction (in other words, $\tilde U_i$ is \textit{strongly Lipschitz contractible} to $\hat q_i$ in the sense of \cite[1.4]{MY:stab}).
This is trivial for the second flow $\Psi_i$, but not for the first (modified) flow $\hat\Phi_i$.
To see this for the first flow, we use the $(-\lambda)$-concavity and $1$-Lipschitz continuity of $\tilde h_i$ to obtain
\begin{align*}
d\tilde h_i(\uparrow_x^{\hat q_i})&\ge\frac{\tilde h_i(\hat q_i)-\tilde h_i(x)+\lambda|\hat q_ix|^2/2}{|\hat q_ix|}\\
&\ge\frac{\tilde h_i(q_i)-|q_i\hat q_i|-\tilde h_i(x)+\lambda|\hat q_ix|^2/2}{|\hat q_ix|}\ge-\frac{|q_i\hat q_i|}{|\hat q_ix|}+\lambda\frac{|\hat q_ix|}2.
\end{align*}
Since $|q_i\hat q_i|\to 0$ as $i\to\infty$, the last term is greater than $\lambda r/50$ outside $\tilde B(\hat q_i,r/20)$.
This means that $\tilde h_i$ is $(\lambda r/50)$-regular in the direction to $\hat q_i$ outside $\tilde B(\hat q_i,r/20)$.
Furthermore, the modified flow $\hat\Phi_i$ fixes $\tilde B(\hat q_i,r/20)$.
By Lemma \ref{lem:reg} the distance to $\hat q_i$ is nonincreasing under this modified flow.
\end{rem}

\section{Quantitative good coverings}\label{sec:good}

In this section, we show the existence of a quantitative version of the good covering of an Alexandrov space introduced in \cite{MY:good}.
The proof is based only on the properties established in Theorem \ref{thm:conv} and we no longer use the Gromov--Hausdorff compactness of the class of Alexandrov spaces.
Fix $n$, $D$, and $v$ as before.

\begin{thm}\label{thm:good}
There is $C=C(n,D,v)>0$ satisfying the following:
for any $M\in \mathcal A(n,D,v)$ and $0<\epsilon<C^{-1}$, there exists a finite open covering $\mathcal U=\{ U_j\}_{j=1}^{N}$ of $M$ such that
\begin{enumerate}
\item any $\epsilon$-ball in $M$ is contained in some $U_j$;
\item for each $j$, the number of $U_i$ intersecting $U_j$ is at most $C$;
\item every nonempty intersection of $U_j$ is $(C,\epsilon)$-Lipschitz (strongly) contractible;
\item for any $M'\in\mathcal A(n,D,v)$ that is $(C^{-1}\epsilon)$-close to $M$ with respect to the Gromov--Hausdorff distance, there is a lift $\mathcal U'$ of $\mathcal U$ to $M'$ that satisfies the same properties as (1)--(3) and has the same nerve as $\mathcal U$.
\end{enumerate}
\end{thm}

The last statement on the nerve means that a collection of $U_j$ has nonempty intersection if and only if so does its lift (see Section \ref{sec:main1} for the nerve).

\begin{rem}
The covering number $N$ depends on $\epsilon$, while the multiplicity of the covering is uniformly bounded independent of $\epsilon$ as in the condition (2).
\end{rem}

As we will see in the following proof, $U_j$ is actually a superlevel set of a strictly concave function constructed in Theorem \ref{thm:conv}.
Excluding the stability condition (4), we define

\begin{dfn}\label{dfn:good}
For a general metric space $M$, a (not necessarily finite) open covering $\mathcal U$ satisfying the conditions (1)--(3) is called a \textit{$(C,\epsilon)$-good covering} of $M$.
\end{dfn}

\begin{proof}[Proof of Theorem \ref{thm:good}]
Throughout the proof, we will denote by $C_0$ the constant $C$ of Theorem \ref{thm:conv}.
Take a maximal $\epsilon/2$-discrete set $\{p_j\}_{j=1}^N$ of $M$.
For each $j$ and $0\le t\le C_0^{-1}$, let
\[U_j(t)=\{h_j>-t\epsilon\}\]
be a $(C_0,\epsilon)$-Lipschitz contractible domain containing $B(p_j,2\epsilon)$, where $h_j$ is a $1$-Lipschitz ($-(C_0\epsilon)^{-1}$)-concave function as in Theorem \ref{thm:conv}(1).
Similarly, we define its lift $U_j'(t)$ to $M'$ by using $h_j'$ of Theorem \ref{thm:conv}(3).

We prove that there exists
\[0\le t_j\le C_0^{-1}\]
depending on $j$ such that the coverings $\{U_j(t_j)\}_j$ and $\{U_j'(t_j)\}_j$ satisfy the conditions (1)--(4) above.
Actually the conditions (1) and (2) are satisfied for an arbitrary choice of $t_j$.
(1) follows from the facts that $\{p_j\}_j$ is a maximal $\epsilon/2$-discrete set and that $U_j(t)$ contains the $2\epsilon$-ball around $p_j$.
(2) follows from these facts plus Remark \ref{rem:diam} and the Bishop--Gromov inequality.

To find $t_j$ satisfying (3) and (4), we show the following claim.

\begin{clm}\label{clm:good}
There exist $0\le t_j\le C_0^{-1}$ and a constant $0<c_0\ll C_0^{-1}$ depending only on $n$, $D$, and $v$ satisfying the following:
\begin{enumerate}
\item[(i)] any nonempty intersection $U$ of $U_j(t_j)$ contains the $(c_0\epsilon)$-ball around its peak;
\item[(ii)] the same holds for $U_j'(t_j)$, and moreover, $\{U_j'(t_j)\}_j$ has the same nerve as $\{U_j(t_j)\}_j$.
\end{enumerate}
\end{clm}

Note that $U$ is again a superlevel set of a strictly concave function $h$ defined as the minimum of $h_j+t_j\epsilon$.
The \textit{peak} of $U$ is the unique maximum point of $h$.

\begin{proof}
\step{Step 1}
We first define $t_j$ and $c_0$ by the following subclaim.
The proof is a simple application of the Bishop--Gromov inequality.

\begin{subclm}\label{subclm:good}
There exist $0\le t_j\le C_0^{-1}$ and $0<c_0\ll C_0^{-1}$ such that the coverings
\[\{U_j(t_j-c_0)\}_j,\quad\{U_j(t_j+c_0)\}_j\]
have the same nerve, where $c_0$ depends only on $n$, $D$, and $v$.
\end{subclm}

\begin{proof}
We prove by induction on $k$ the following partial statement:
\begin{quote}
There exist $0\le t_j^k\le C_0^{-1}$ ($1\le j\le k$) and $0<c_0^k\ll C_0^{-1}$ depending on $k$ such that $\{U_j(t_j^k-c_0^k)\}_{j=1}^k$ and $\{U_j(t_j^k+c_0^k)\}_{j=1}^k$ have the same nerve.
\end{quote}
The desired $t_j$ and $c_0$ would be obtained as $t_j^N$ and $c_0^N$, where $N$ is the covering number (but see also the last paragraph of the proof).

The base case $k=1$ is trivial.
Suppose the above claim holds for $k$.
Consider a one-parameter family of subcoverings
\[\mathcal U_{k+1}(t):=\{U_j(t_j^k+c_0^kt)\}_{j=1}^k\cup\{U_{k+1}(t)\},\]
where $0\le t\le C_0^{-1}$ (we may assume $t_j^k+c_0^kC_0^{-1}\le C_0^{-1}$ so that $U_j(t_j^k+c_0^kC_0^{-1})$ makes sense).
Clearly the nerve of the above subcovering is nondecreasing in $t$ with respect to the inclusion relation.
Recall that the multiplicity of the largest covering $\{U_j(C_0^{-1})\}_j$ is uniformly bounded by a constant depending only on $n$, $D$, and $v$ (see the proof of Theorem \ref{thm:good}(2) before Claim \ref{clm:good}).
Hence one can find $0\le t_{k+1}\le C_0^{-1}$ and $0<c_0'\ll C_0^{-1}$ such that the coverings
\[\mathcal U_{k+1}(t_{k+1}-c_0'),\quad\mathcal U_{k+1}(t_{k+1}+c_0')\]
have the same nerve, where $c_0'$ depends only on $n$, $D$, and $v$.
Now we define
\[t_j^{k+1}:=t_j^k+c_0^k t_{k+1},\quad t_{k+1}^{k+1}:=t_{k+1},\quad c_0^{k+1}:=c_0^kc_0',\]
for $1\le j\le k$.
These satisfy the partial statement for $k+1$.

In the above argument, the number of the induction steps is $N$, the covering number, which depends on $\epsilon$ and is not uniformly bounded in terms of $n$, $D$, and $v$.
However, since the multiplicity of the largest covering $\{U_j(C_0^{-1})\}_j$ is uniformly bounded, one can modify the above argument so that the number of the actual induction steps applied to each $U_j$ is uniformly bounded.
Consequently, we obtain a constant $c_0$ depending only on $n$, $D$, and $v$, as stated in Subclaim \ref{subclm:good}.
\end{proof}

\step{Step 2}
We now prove (i) and (ii) of Claim \ref{clm:good}.

To prove (i), let $U$ be the nonempty intersection of $U_i(t_i)$ and $p\in U$ its peak, where $i$ runs over a subset of $\{1,\dots,N\}$.
Recall that
\[U=\{h>0\},\quad h=\min_i\{h_i+t_i\epsilon\},\]
and $p$ is the unique maximum point of $h$.
By Subclaim \ref{subclm:good}, $U_i(t_i-c_0)=\{h_i>-(t_i-c_0)\epsilon\}$ also have nonempty intersection, and thus
\[h(p)>c_0\epsilon.\]
Since $h$ is $1$-Lipschitz, this implies $B(p,c_0\epsilon)\subset U$, as desired.

To prove (ii), suppose $M'$ is $(C^{-1}\epsilon)$-close to $M$ as in the assumption of Theorem \ref{thm:good}, where $C$ is large enough.
Let $U$ and $h$ be as above.
Let $U'$ and $h'$ denote the lifts of $U$ and $h$ to $M'$, respectively, that is, $U'=\{h'>0\}$.
From the construction of $h$ in Theorem \ref{thm:conv}, we see that $h'$ is $(C^{-1}\epsilon)$-close to $h$ via the Gromov--Hausdorff approximation.
Therefore, if $p'\in U'$ is the unique maximum point of $h'$, we have
\[h'(p')>c_0\epsilon/2,\]
provided $C^{-1}\ll c_0$.
Since $h'$ is $1$-Lipschitz, this implies $B(p',c_0\epsilon/2)\subset U'$.
Replacing $c_0$ with $c_0/2$, we obtain the first half of (ii).

The second half of (ii) immediately follows from Subclaim \ref{subclm:good} and the fact that $h'$ is $(C^{-1}\epsilon)$-close to $h$, providing $C^{-1}\ll c_0$.
\end{proof}

Now we finish the proof of Theorem \ref{thm:good}.
In what follows, $C$ denotes various positive constants depending only on $n$, $D$, and $v$.
To prove the condition (3), let $U$ be any nonempty intersection of $U_j(t_j)$ with peak $p$, defined by a strictly concave function $h$.
By Claim \ref{clm:good}(i), $U$ contains $B(p,c_0\epsilon)$.
Since $h$ is ($-(C_0\epsilon)^{-1}$)-concave, Lemma \ref{lem:grad} implies that
\[|\nabla h|\ge(2C_0C)^{-1}c_0\]
outside $B(p,C^{-1}c_0\epsilon)$.
Hence the gradient flow of $h$ provides a $(C,\epsilon)$-Lipschitz deformation retraction of $U$ into $B(p,C^{-1}c_0\epsilon)$.
By Theorem \ref{thm:conv}, $B(p,C^{-1}c_0\epsilon)$ is $(C,\epsilon)$-Lipschitz contractible in $B(p,c_0\epsilon)\subset U$.
Therefore $U$ is $(C,\epsilon)$-Lipschitz contractible in itself.

The remaining properties in the condition (4) follows from Claim \ref{clm:good}(ii) in the same way.
This completes the proof.
\end{proof}

\begin{rem}\label{rem:CAT}
Another important class of metric spaces having such good coverings is CAT spaces with a uniform local doubling constant (see Section \ref{sec:main2} for the doubling property).
Suppose $X$ is a locally $C$-doubling CAT($1$) space.
Take a maximal $\epsilon/2$-discrete set $\{p_j\}_j$ of $X$ and set $U_j:=B(p_j,2\epsilon)$, where $0<\epsilon<(10C)^{-1}$.
Then the same argument as above shows that this covering satisfies the conditions (1) and (2) of Theorem \ref{thm:good} (use the doubling property instead of the Bishop--Gromov inequality).
Moreover, if $\epsilon<\pi/10$, it satisfies the condition (3) since the CAT($1$) property implies that
\begin{itemize}
\item any metric ball $B$ in $X$ of radius less than $\pi/2$ is convex;
\item shortest paths are unique for every pair of points in $B$;
\item the contraction of $B$ to an arbitrary point $p\in B$ along the unique shortest paths is uniformly Lipschitz.
\end{itemize}
In particular, Theorem \ref{thm:nerve} in the next section can also be applied to such CAT spaces.
Furthermore, it is possible to obtain the stability condition (4) by considering $U_j(t_j):=B(p_j,t_j\epsilon)$ instead, for suitably chosen $2\le t_j\le 3$.
The value $t_j$ is defined by induction on $j$ so as to satisfy Subclaim \ref{subclm:good}, where $c_0$ now depends on the doubling constant.
Note that the proof of Subclaim \ref{subclm:good} only uses the doubling property.
\end{rem}

\section{First proof of Theorem \ref{thm:main}}\label{sec:main1}

In this section, we give the first proof of Theorem \ref{thm:main} using good coverings.
Since the good coverings of Alexandrov spaces satisfy the stability property as in Theorem \ref{thm:good}(4), it suffices to prove the following theorem (for more details, see the end of this section).

\begin{thm}\label{thm:nerve}
For any $C>0$ there is $\tilde C>0$ satisfying the following:
Let $M$ be a metric space with a $(C,\epsilon)$-good covering $\mathcal U$ in the sense of Definition \ref{dfn:good}.
Then $M$ is $(\tilde C,\epsilon)$-Lipschitz homotopy equivalent to the $\epsilon$-geometric realization of the nerve of $\mathcal U$.
\end{thm}

Let us first explain our terminology.
The \textit{nerve} of a locally finite covering $\mathcal U=\{U_j\}_j$ is a simplicial complex with vertices $\{v_j\}_j$ such that a collection of vertices spans a simplex if and only if the corresponding $U_j$ have nonempty intersection.
We denote it by $\mathcal N=\mathcal N_\mathcal U$.
Note that the condition (2) of Theorem \ref{thm:good} implies that the nerve of a $(C,\epsilon)$-good covering has uniform local finiteness and in particular has uniformly bounded dimension.

Let $K$ be a locally finite simplicial complex.
For simplicity we assume that $K$ has only finitely many vertices $\{v_j\}_{j=1}^N$, but this is not necessary because of the local finiteness.
For $\epsilon>0$, the \textit{$\epsilon$-geometric realization} $|K|=|K|_\epsilon$ of $K$ is defined as follows.
We identify $v_j$ with the $j$-th standard vector of norm $\epsilon$ in Euclidean space, i.e.,
\[v_j=(0,\dots,\overset j\epsilon,\dots,0)\in\mathbb R^N.\]
The set $|K|$ is a subset of the convex hull of $\{v_j\}_j$ corresponding to $K$.
More precisely, a function $\theta:\{v_j\}_j\to[0,1]$ defines an element $\sum_j\theta(v_j)v_j$ of $|K|$ if it satisfies
\begin{enumerate}
\item $\sum_j \theta(v_j)=1;$
\item $\supp\theta$ defines a simplex of $K$ (also denoted by $\supp\theta$).
\end{enumerate}
We consider the length metric on $|K|$ induced by the standard Euclidean metric on $\mathbb R^N$.
We identify a simplex $\sigma\in K$ with its closed realization in $|K|$ and denote it by the same symbol $\sigma$.

\begin{conv}\label{conv:C}
In the following proof, we will use the same symbol $C$ to denote various large positive constants.
Whenever this symbol appears, it means that there exists a constant $C$ such that the claim holds (possibly different from $C$ in Theorem \ref{thm:nerve}, but depending only on it as $\tilde C$).
\end{conv}

\begin{proof}[Proof of Theorem \ref{thm:nerve}]
The proof is along the same lines as in \cite[\S3]{MY:lip}, where the second and third authors proved the qualitative version of this theorem (based on the idea of \cite[9.4.15]{Sha}).
In fact all we have to do is repeat the $\epsilon$-scaled version of the previous argument.
We will skip detailed calculations if they were already done in \cite{MY:lip}.

Let $\mathcal U=\{U_j\}_j$ be a $(C,\epsilon)$-good covering of $M$.
Let $\mathcal N=\mathcal N_{\mathcal U}$ denote its nerve and $|\mathcal N|=|\mathcal N|_\epsilon$ its $\epsilon$-geometric realization.
Note that the condition (2) of Theorem \ref{thm:good} implies that the dimension of $\mathcal N$ is at most $C$.
This fact will be used throughout the proof.
We prove that there exist $C$-Lipschitz maps
\[\Theta:M\to|\mathcal N|,\quad\zeta:|\mathcal N|\to M\]
such that $\zeta\circ\Theta$ and $\Theta\circ\zeta$ are $(C,\epsilon)$-Lipschitz homotopic to the identities of $M$ and $|\mathcal N|$, respectively.

\step{Step 1}
We first define the map $\Theta$ by using the partition of unity subordinate to $\mathcal U$.
Let $f_j=\dist_{M\setminus U_j}$ be the distance function to the complement of $U_j$ and set
\[\xi_j(x):=\frac{f_j(x)}{\sum_if_i(x)}\]
for $x\in M$ (the above $f_j$ is simpler than the one in \cite{MY:lip}, but this is not an essential change).
Then

\begin{clm}\label{clm:xi}
$\xi_j$ is $C/\epsilon$-Lipschitz.
\end{clm}

\begin{proof}
By the condition (1) of Theorem \ref{thm:good}, we have $\sum_if_i(x)\ge\epsilon$ for any $x\in M$.
A direct calculation using the triangle inequality shows
\begin{align*}
&\left|\xi_j(x)-\xi_j(y)\right|\\
&\le\left|\frac{f_j(x)-f_j(y)}{\sum_if_i(x)}\right|+\left|\frac{\sum_if_i(x)-\sum_if_i(y)}{\sum_if_i(x)}\right|\frac{f_j(y)}{\sum_if_i(y)}.
\end{align*}
Since $f_j$ is $1$-Lipschitz, the first term is bounded above by $|xy|/\epsilon$.
Similarly, since $\dim\mathcal N\le C$, the second term is bounded above by $C|xy|/\epsilon$.
\end{proof}

We define $\Theta:M\to |\mathcal N|$ by
\[\Theta(x):=\sum_j\xi_j(x)v_j,\]
where $v_j$ is the $j$-th standard vector of norm $\epsilon$.
Then Claim \ref{clm:xi} together with $\dim\mathcal N\le C$ implies that $\Theta$ is $C$-Lipschitz.

\step{Step 2}
Next we construct larger spaces $\mathcal D$ and $\mathcal M$ together with the following three $C$-bi-Lipschitz embeddings
\[
\begin{CD}
M @>\tau >> \mathcal D\\
@.
@VV{\iota}V \\
|\mathcal N| @>>\Psi> \mathcal M.
\end{CD}
\]
We show that their images are $(C,\epsilon)$-Lipschitz strong deformation retracts of the target spaces.
Only the deformation retraction for the $\iota$-image is nontrivial, which is deferred to Step 3.

For a simplex $\sigma\in\mathcal N$, we denote by $U_\sigma$ the corresponding intersection of $U_j$.
We define $\mathcal D$ by
\begin{align*}
\mathcal D&:=\{(\theta, x)\in |\mathcal N|\times M \,|\, x\in U_{\supp\theta}\} \\
 & =\bigcup_{\sigma\in\mathcal N} \sigma\times U_\sigma\subset |\mathcal N|\times M.
\end{align*}
As above we will use $\theta$ and $x$ to denote the $|\mathcal N|$- and $M$-coordinates, respectively.
We also denote the projections by
\[p:\mathcal D\to|\mathcal N|,\quad q:\mathcal D\to M.\]
The $C$-bi-Lipschitz embedding $\tau:M\to\mathcal D$ is given by
\[\tau(x):=(\Theta(x),x),\]
whose inverse is the restriction of $q$.

\begin{clm}\label{clm:tau}
$\tau(M)$ is a $(C,\epsilon)$-Lipschitz strong deformation retract of $\mathcal D$.
\end{clm}

\begin{proof}
The desired homotopy $H:\mathcal D\times[0,\epsilon]\to\mathcal D$ is given by
\[H(\theta,x,s):=((s/\epsilon)\Theta(x)+(1-s/\epsilon)\theta,x)\]
for $s\in[0,\epsilon]$.
Since each simplex has size $\epsilon$, this map is $C$-Lipschitz.
\end{proof}

We define $\mathcal M$ as the mapping cylinder of $p$ with height $\epsilon$:
\begin{align*}
\mathcal M=\mathcal M_\epsilon(p)&:= \mathcal D\times [0,\epsilon]\cup |\mathcal N|/(\theta, x,\epsilon)\sim \theta\\
&=\bigcup_{\sigma\in K} \sigma\times K(U_{{\sigma}})\subset |\mathcal N|\times K(M).
\end{align*}
Here $K(M)=K_\epsilon(M):=M\times [0,\epsilon]/M\times \epsilon$ denotes the Euclidean cone of height $\epsilon$ equipped with the metric
\begin{align*}
&\left|[x,t],[x',t']\right|^2\\
&:=(\epsilon- t)^2+(\epsilon-t')^2-2(\epsilon-t)(\epsilon-t')\cos\min\{\pi,|xx'|\},
\end{align*}
for $[x,t],[x',t']\in K(M)$ (the reason for this ``reversed'' metric is that the height corresponds to the time of homotopy later).
We consider the metric of $\mathcal M$ induced from the product metric of $|\mathcal N|\times K(M)$.

We define
\[\iota:\mathcal D\to\mathcal M,\quad\Psi:|\mathcal N|\to\mathcal M\]
by natural isometric embeddings.
We also define
\[\Psi':\mathcal M\to|\mathcal N|\]
by a natural map induced by the projection $p:\mathcal D\to|\mathcal N|$, whose restriction to $\Psi(|\mathcal N|)$ is the inverse of $\Psi$.
More precisely, 
\[\iota(\theta,x):=(\theta,x,0),\quad\Psi(\theta):=(\theta,v_M),\quad \Psi'(\theta,[x,t]):=\theta\]
where $v_M$ is the vertex of the cone $K(M)$.

Note that we have the following commutative diagram:

\begin{equation}\label{eq:Theta}
\begin{CD}
M @>\tau>> \mathcal D \\
@V\Theta VV @VV\iota V \\
|\mathcal N| @<<\Psi'< \mathcal M.
\end{CD}
\end{equation}

\begin{clm}\label{clm:Psi}
$\Psi(|\mathcal N|)$ is a $(C,\epsilon)$-Lipschitz strong deformation retract of $\mathcal M$.
\end{clm}

\begin{proof}
The desired homotopy $F:\mathcal M\times[0,\epsilon]\to\mathcal M$ is given by
\[F(\theta,[x,t],s):=(\theta, [x, (1-s/\epsilon)t+s])\]
for $s\in[0,\epsilon]$.
Since the cone has height $\epsilon$, this map is $C$-Lipschitz.
The details are left to the reader (see the proof of \cite[3.3]{MY:lip}).
\end{proof}

\step{Step 3}
Here we prove

\begin{clm}\label{clm:Phi}
There exists a $(C,\epsilon)$-Lipschitz strong deformation retraction $\Phi$ of $\mathcal M$ to $\iota(\mathcal D)$.
\end{clm}

We will construct $\Phi$ by reverse induction on the skeleta of $\mathcal N$.
Since the geometric realization $|\mathcal N|$ has the length metric, the problem is reduced to the following claim on each simplex, exactly as in \cite[\S3]{MY:lip}.

\begin{subclm}\label{subclm:Phi}
For any $\sigma\in\mathcal N$, there exists a $(C,\epsilon)$-Lipschitz strong deformation retraction $\Phi_\sigma$ of $\sigma\times K(U_{\sigma})$ to $\sigma\times U_{\sigma}\times 0\cup\partial\sigma\times K(U_{\sigma})$.
\end{subclm}

The desired map $\Phi$ is obtained by gluing $\Phi_\sigma$ by reverse induction on the skeleta.
Since $\dim\mathcal N\le C$, the Lipschitz constant and the time of the homotopy remain uniformly bounded above.

\begin{rem}
There was a minor mistake in the proof of the corresponding claim in \cite{MY:lip} which might confuse the reader.
In the proofs of Sublemma 3.6 and Claim 3.8 in \cite{MY:lip}, we assumed that the the Lipschitz contraction $\varphi$ satisfies $\varphi(\cdot,t)\equiv p$ for $t\ge L/2$.
However, this choice of a constant was not sufficient for the proofs.
A suitable choice is, for example, $\varphi\equiv p$ for $t\ge L/10$.
Actually one may choose an arbitrary small constant $0<c\ll 1$ so that $\varphi\equiv p$ for $t\ge cL$.
\end{rem}

\begin{proof}
The basic idea is that we first contract a neighborhood of the vertex of the cone, then the rest is almost a product on which Lipschitz continuity is easy to prove.

To simplify the notation, we will now omit the subscript $\sigma$.
Let
\[\varphi:U\times[0,\epsilon] \to U\]
be a $(C,\epsilon)$-Lipschitz contraction with center $p\in U$.
We may assume that $\varphi(\cdot,t)\equiv p$ for all $t\ge\epsilon/10$.
We divide the problem into two cases, depending on the dimension of $\sigma$.

\step{Case 1}
Suppose $\dim\sigma=0$.
Define a retraction $r:K(U)\to U\times 0$ by
\[r([x,t]):=[\varphi(x,t),0],\]
which is well-defined for $t=\epsilon$.

Let $g:[0,\epsilon]\to[0,\epsilon]$ be a $C$-Lipschitz function such that
\begin{equation*}
g(s)=
\begin{cases}
\epsilon & s\le\epsilon/3\\
0 & s=\epsilon.
\end{cases}
\end{equation*}
Define $\Phi:K(U)\times[0,\epsilon]\to K(U)$ by
\[\Phi([x,t],s):=[\varphi(x,(s/\epsilon)t),(g(s)/\epsilon)t]\]
for $s\in[0,\epsilon]$.
Then it gives a $(C,\epsilon)$-Lipschitz homotopy between the identity and $r$.
The details are left to the reader (see the proof of \cite[3.8]{MY:lip}).

\step{Case 2}
Suppose $\dim\sigma\ge 1$.
Let
\[\rho=(\psi,u):\sigma\times [0,\epsilon]\to\sigma\times 0\cup\partial\sigma\times [0,\epsilon]\]
be the radial projection from $(\theta_*,2\epsilon)$ in $\sigma\times[0,2\epsilon]$, where $\theta_*$ is the barycenter of $\sigma$.
Since $\sigma$ is a simplex of size $\epsilon$, $\rho$ is a $C$-Lipschitz retraction.

Define a retraction $r:\sigma\times K(U)\to\sigma\times U\times0\cup\partial\sigma\times K(U)$ by
\[r(\theta,[x,t]):=(\psi(\theta,t),[\varphi(x,t-u(\theta,t)),w(\theta,t)]),\]
where $w:\sigma\times[0,\epsilon]\to[0,\epsilon]$ is defined as follows.
Set $\Omega_0,\Omega_1\subset\sigma\times[0,\epsilon]$ by
\[\Omega_0:=u^{-1}[0,\epsilon/10],\quad\Omega_1:=u^{-1}[\epsilon/2,\epsilon]\]
and let $s_i$ be the distance functions to $\Omega_i$ ($i=0,1$) defined on $\sigma\times[0,\epsilon]$.
We then define a function $w$ on $\sigma\times[0,\epsilon]$ by 
\[w(\theta,t):=\frac{s_1}{s_0+s_1}u+\frac{s_0}{s_0+s_1}t\]
so that $w=u$ on $\Omega_0$ and $w=t$ on $\Omega_1$.
Since $\sigma$ is of size $\epsilon$, $w$ is $C$-Lipschitz.
Note that $r$ is well-defined for $t=\epsilon$.

Let $\mu,\nu:[0,\epsilon]\to[0,\epsilon]$ be $C$-Lipschitz functions such that
\begin{equation*}
\mu(s)=
\begin{cases}
0 & s\le\epsilon/2\\
\epsilon & s\ge2\epsilon/3,
\end{cases}
\quad
\nu(s)=
\begin{cases}
0 & s\le2\epsilon/3\\
\epsilon & s\ge3\epsilon/4.
\end{cases}
\end{equation*}
Define $\Phi:\sigma\times K(U)\times[0,\epsilon]\to\sigma\times K(U)$ by
\begin{align*}
\Phi(\theta,[x,t],s):=((&1-(s/\epsilon))\theta+(s/\epsilon)\psi,\\
&[\varphi(x,(\mu/\epsilon)(t-u)),((1-(\nu/\epsilon))t+(\nu/\epsilon)w)])
\end{align*}
for $s\in[0,\epsilon]$.
Then it gives a $(C,\epsilon)$-Lipschitz homotopy between the identity and $r$.
The details are left to the reader (see the proof of \cite[3.6]{MY:lip}).
\end{proof}

\step{Step 4}
The inverse map $\zeta:|\mathcal N|\to M$ is defined by the following commutative diagram:
\begin{equation}\label{eq:zeta}
\begin{CD}
M @<q << \mathcal D \\
@A\zeta AA @AA\bar\Phi A \\
|\mathcal N| @>>\Psi> \mathcal M,
\end{CD}
\end{equation}
where $\bar\Phi:=\Phi(\cdot,\epsilon)$ and $\mathcal D$ is identified with its $\iota$-image.
Now the statement follows from \eqref{eq:Theta}, \eqref{eq:zeta}, Claims \ref{clm:tau}, \ref{clm:Psi}, and \ref{clm:Phi}.
\end{proof}

\begin{rem}\label{rem:nerve}
By construction, for any $x\in M$ we have $x\in U_{\supp\Theta(x)}$.
Furthermore, for any $\theta\in|\mathcal N|$ there exists a vertex of $\supp\theta$ such that the corresponding $U_j$ contains $\zeta(\theta)$.
Therefore $\zeta\circ\Theta(x)$ is contained in some $U_j$ containing $x$.
\end{rem}

We are now in a position to give the first proof of Theorem \ref{thm:main}.

\begin{proof}[Proof of Theorem \ref{thm:main}]
By replacing the constant $C$ in the statement, we may assume that $M,M'\in\mathcal A(n,D,v)$ are $(C^{-1}\epsilon)$-close with respect to the Gromov--Hausdorff distance, where $0<\epsilon<C^{-1}$.
In what follows, we will abuse $C$ to denote various large constants depending only on $n$, $D$, and $v$.
By Theorem \ref{thm:good}, there exist a $(C,\epsilon)$-good covering $\mathcal U$ of $M$ and a corresponding $(C,\epsilon)$-good covering $\mathcal U'$ of $M'$ with the same nerve as $\mathcal U$.
By Theorem \ref{thm:nerve}, $M$ and $M'$ are $(C,\epsilon)$-Lipschitz homotopy equivalent to the $\epsilon$-geometric realization of the nerve.
Therefore $M$ and $M'$ are $(C,\epsilon)$-Lipschitz homotopy equivalent.
More precisely, it is given by
\[M\overset\Theta{\underset\zeta\rightleftarrows}|\mathcal N|=|\mathcal N'|\overset{\Theta'}{\underset{\zeta'}\leftrightarrows} M'.\]
Since $\mathcal U'$ is a lift of $\mathcal U$, an observation similar to Remark \ref{rem:nerve} shows that the $\zeta'\circ\Theta$ and $\zeta\circ\Theta'$ are $C\epsilon$-close to the original Gromov--Hausdorff approximations between $M$ and $M'$.
In particular they are also $C\epsilon$-approximations.
This completes the proof.
\end{proof}

\section{Second proof of Theorem \ref{thm:main}}\label{sec:main2}

In this section, we give the second proof of Theorem \ref{thm:main} using only local contractibility.
We prove Theorem \ref{thm:Peter}, the Lipschitz version of Petersen's stability theorem \cite[Theorem A]{Peter}.
The proof is in part a generalization of the previous argument of good coverings.

Let us first explain our terminology in Theorem \ref{thm:Peter}.
Fix $C>0$.
A metric space is \textit{locally $C$-Lipschitz contractible} if any $\epsilon$-ball is contained in some $(C,\epsilon)$-Lipschitz contractible domain, where $0<\epsilon<C^{-1}$.
A metric space is \textit{locally $C$-doubling} if any $\epsilon$-ball is covered by at most $C$ balls of radius $\epsilon/2$, where $0<\epsilon<C^{-1}$.
These conditions correspond to the assumptions on the contractibility function and the covering dimension, respectively, in the original theorem of Petersen \cite{Peter}.

For Alexandrov spaces, the local $C$-Lipschitz contractibility was already established in Theorem \ref{thm:conv}, and the local $C$-doubling condition follows from the Bishop--Gromov inequality.
Therefore Theorem \ref{thm:Peter} immediately implies Theorem \ref{thm:main}.
As explained in Remark \ref{rem:CAT}, CAT spaces with a uniform local doubling constant also satisfy the assumption of Theorem \ref{thm:Peter}.

The proof of Theorem \ref{thm:Peter} is along almost the same lines as that of \cite[Theorem A]{Peter}.
Only the existence of a domination by a polyhedron (Proposition \ref{prop:dom}), which was used in \cite{Peter} without proof by leaving it to \cite[Ch.\ IV]{Hu}, requires an additional argument.
However, this is a minor modification of the previous section.

In the following proof, we will use Convention \ref{conv:C} as in the previous section.
Furthermore we will use

\begin{conv}\label{conv:C'}
In this section $\tilde C$ denotes a large positive constant depending only on $C$ and $\epsilon$ denotes a small positive number less than some constant depending only on $C$.
More precisely, whenever these symbols appear, it means:
\begin{quote}
for any $C>0$ there exists $\tilde C>0$ such that for any $0<\epsilon<\tilde C^{-1}$ the following holds.
\end{quote}
To be precise, all the statements in this section should be preceded by this sentence, but we will omit it for simplicity.
\end{conv}

\subsection{Lipschitz maps from polyhedra}

We will use the same notation and terminology as in the previous section.
For a (locally finite) simplicial complex $K$, we denote by $|K|=|K|_\epsilon$ its $\epsilon$-geometric realization.
We identify a simplex $\sigma\in K$ with its closed realization in $|K|$ and denote it by $\sigma$.

Furthermore, if $L$ is a subcomplex $K$, we also identify $|L|$ with a subset of $|K|$ (equipped with the metric of $|K|$).
We denote by $K^0$ the $0$-skeleton of $K$, i.e., the set of vertices of $K$.

We first prove the Lipschitz version of \cite[\S2 Main Lemma]{Peter}.
The idea of the proof is similar to that of Subclaim \ref{subclm:Phi}.

\begin{prop}\label{prop:simp}
Let $M$ be a locally $C$-Lipschitz contractible space and $K$ a simplicial complex of dimension $\le C$ with subcomplex $L$.
Let $|K|=|K|_\epsilon$ denote the $\epsilon$-geometric realization.
Suppose
\[f:|K^0|\cup|L|\to M\]
is $C$-Lipschitz.
If $\epsilon$ is small enough, then there exists a $\tilde C$-Lipschitz extension $\tilde f:|K|\to M$ of $f$.
\end{prop}

\begin{proof}
The proof is almost the same as that of \cite[\S2 Main Lemma]{Peter}.
However, we need to be more careful about Lipschitz constants.
Since the geometric realization has the length metric, it suffices to check the Lipschitz continuity on each simplex.

The map $\tilde f$ is constructed by induction on skeleta.
On $|K^0|\cup|L|$ we define $\tilde f:=f$, which includes the base case.
Let $\sigma\in K\setminus(K^0\cup L)$ be a simplex.
By the induction hypothesis, we may assume that $f$ is already defined on $\partial\sigma$.
By the Lipschitz continuity of $f$, the diameter of $f(\partial\sigma)$ is bounded above by $C\epsilon$.
By assumption, there exists a $(C,\epsilon)$-Lipschitz contractible domain $U$ containing $f(\partial\sigma)$, provided $\epsilon$ is small enough.
We may assume that the $(C,\epsilon)$-Lipschitz contraction
\[\varphi:U\times[0,\epsilon]\to U\]
satisfies $\varphi(\cdot,t)\equiv p$ for all $t\ge\epsilon/2$, where $p$ is the center of the contraction of $U$.

Let $x_\ast$ be the barycenter of $\sigma$.
Any point $x\in\sigma\setminus\{x_\ast\}$ is uniquely represented as a convex combination $t(x)x_\ast+(1-t(x))y(x)$, where $y(x)\in\partial\sigma$ and $0\le t(x)<1$ (we also define $t(x_\ast):=1$).
Observe that the map $x\mapsto(y(x),\epsilon t(x))$ is a $C$-bi-Lipschitz homeomorphism between $t^{-1}[0,1/2]\subset\sigma$ and the product space $\partial\sigma\times[0,\epsilon/2]$.
Therefore if we define $\tilde f$ on $\sigma$ by
\[\tilde f(x):=\varphi(f\circ y(x),\epsilon t(x)),\]
then it gives a $C$-Lipschitz extension of $f$ onto $\sigma$ (well-defined for $x_\ast$).
This completes the induction step.
Note that the Lipschitz constant gets larger at each induction step, but the number of the induction steps is uniformly bounded above by assumption.
\end{proof}

\begin{cor}\label{cor:simp}
Let $M$ be a locally $C$-Lipschitz contractible space and $K$ a simplicial complex of dimension $\le C$.
Let $|K|=|K|_\epsilon$ denote the $\epsilon$-geometric realization.
Suppose $C$-Lipschitz maps
\[f_i:|K|\to M,\quad i=0,\epsilon\]
are $\epsilon$-close in the uniform distance.
If $\epsilon$ is small enough, then there exists a $(\tilde C,\epsilon)$-Lipschitz homotopy between $f_0$ and $f_\epsilon$.
\end{cor}

\begin{proof}
The product space $|K|\times[0,\epsilon]$ admits a natural triangulation whose vertices are exactly those of $|K|\times\{0,\epsilon\}$.
Furthermore, $|K|\times[0,\epsilon]$ with the product metric is $C$-bi-Lipschitz homeomorphic to the $\epsilon$-geometric realization of this triangulation.
The assertion then follows from Proposition \ref{prop:simp} by regarding $|K|\times\{0,\epsilon\}$ as $|L|$.
Note that the assumption that $f_0$ and $f_\epsilon$ are $\epsilon$-close implies that the glued map $f_0\cup f_\epsilon$ defined on $|K|\times\{0,\epsilon\}$ is $C$-Lipschitz.
\end{proof}

\subsection{Lipschitz dominations by polyhedra}\label{sec:dom}

The contents of this subsection have no counterparts in \cite{Peter}.

Let $M$ be a metric space that is locally $C$-Lipschitz contractible and locally $C$-doubling.
For sufficiently small $\epsilon>0$ as in Convention \ref{conv:C'}, take a maximal $\epsilon/2$-discrete set $\{p_j\}_j$ in $M$.
Instead of a good covering, we will consider an open covering $\mathcal U=\{B(p_j,2\epsilon)\}_j$ consisting of metric balls.
Let $\mathcal N=\mathcal N_\mathcal U$ denote its nerve and $|\mathcal N|=|\mathcal N|_\epsilon$ its $\epsilon$-geometric realization.

Since $\{p_j\}_j$ is a maximal $\epsilon/2$-discrete set, any $\epsilon$-ball in $M$ is contained in some $B(p_j,2\epsilon)$.
Furthermore, the local $C$-doubling condition implies that $\mathcal N$ has uniform local finiteness.
In other words, this covering still satisfies the conditions (1) and (2) of Theorem \ref{thm:good} (but not (3)).

In this subsection we prove

\begin{prop}\label{prop:dom}
$M$ is $(\tilde C,\epsilon)$-Lipschitz dominated by $|\mathcal N|$, i.e., there exist $\tilde C$-Lipschitz maps
\[\Theta:M\to |\mathcal N|,\quad\zeta:|\mathcal N|\to M\]
such that $\zeta\circ\Theta$ is $(\tilde C,\epsilon)$-Lipschitz homotopic to the identity $1_M$.
\end{prop}

\begin{rem}
Unlike the case of a good covering, $\Theta\circ\zeta$ is not generally homotopic to $1_{|\mathcal N|}$, due to the lack of the condition (3) of Theorem \ref{thm:good}.
See also Remark \ref{rem:tau'} below.
\end{rem}

\begin{proof}
The proof is along almost the same lines as Theorem \ref{thm:nerve}.
Instead of considering intersections of metric balls defined by the nerve $\mathcal N$, we will construct a family of $(C,\epsilon)$-Lipschitz contractible domains corresponding to $\mathcal N$, which enables us to prove only the one-sided homotopy equivalence.
Once such a family is constructed, the rest of the proof is the same as the previous one, so we will omit it.

\step{Step 1}
As before, the map $\Theta$ is defined by using the partition of unity subordinate to $\mathcal U$.
Let $f_j$ be the distance function to the complement of $B(p_j,2\epsilon)$ and set
\[\xi_j(x):=\frac{f_j(x)}{\sum_i f_i(x)}\]
for $x\in M$.
Since $\mathcal U$ still satisfies the conditions (1) and (2) of Theorem \ref{thm:good}, the same argument as in the proof of Claim \ref{clm:xi} shows

\begin{clm}\label{clm:xi'}
$\xi_j$ is $C/\epsilon$-Lipschitz.
\end{clm}

We define a $C$-Lipschitz map $\Theta:M\to |\mathcal N|$ by $\Theta(x):=\sum_j\xi_j(x)v_j$, where $v_j$ is the $j$-th standard vector of norm $\epsilon$.

\step{Step 2}
Next we construct larger spaces $\mathcal D$ and $\mathcal M$ together with three $C$-bi-Lipschitz embeddings
\[
\begin{CD}
M @>\tau >> \mathcal D\\
@.
@VV{\iota}V \\
|\mathcal N| @>>\Psi> \mathcal M.
\end{CD}
\]
This is the only place where an essential modification is required.

We first assign to each simplex $\sigma\in\mathcal N$ a $(C,\epsilon)$-Lipschitz contractible domain $U_\sigma$ in $M$.
This is done by reverse induction on skeleta.
Let $\sigma\in\mathcal N$ be a maximal simplex (with respect to the inclusion relation).
Denote by $B_\sigma$ the corresponding nonempty intersection of $2\epsilon$-balls in $M$.
By the local $C$-Lipschitz contractibility of $M$, there exists a $(C,\epsilon)$-Lipschitz contractible domain $U_\sigma$ containing $B_\sigma$.
Next let $\tau\in\mathcal N$ be a nonmaximal simplex.
Then there exists a $(C,\epsilon)$-Lipschitz contractible domain $U_\tau$ containing
\[B_\tau\cup_{\tau\subset\sigma}U_\sigma,\]
where $\sigma$ contains $\tau$ as a face and $U_\sigma$ is the one given by the induction hypothesis.
In this way we obtain a family $\{U_\sigma\}_{\sigma\in\mathcal N}$ of $(C,\epsilon)$-Lipschitz contractible domains satisfying the following:
\begin{enumerate}
\item $U_\sigma$ contains $B_\sigma$;
\item if $\tau$ is a face of $\sigma$, then $U_\tau$ contains $U_\sigma$.
\end{enumerate}
Note that the Lipschitz constant gets larger at each induction step, but the number of the induction steps is uniformly bounded above by the local $C$-doubling condition.

By using this family, we define $\mathcal D$ in the same way as before:
\begin{align*}
\mathcal D&:=\{(\theta, x)\in |\mathcal N|\times M \,|\, x\in U_{\supp\theta}\} \\
 & =\bigcup_{\sigma\in\mathcal N} \sigma\times U_\sigma\subset |\mathcal N|\times M.
\end{align*}
Here the second equality follows from the condition (2) above.
Let $p:\mathcal D\to|\mathcal N|$ and $q:\mathcal D\to M$ be the projections.
The $C$-bi-Lipschitz embedding $\tau:M\to\mathcal D$ is given by $\tau(x):=(\Theta(x),x)$ with inverse $q$.
Note that the image of $\tau$ is certainly contained in $\mathcal D$ because of the condition (1) above.
As before, we define $\mathcal M$ as the mapping cylinder of $p$ with height $\epsilon$.

The maps $\iota:\mathcal D\to\mathcal M$ and $\Psi:|\mathcal N|\to\mathcal M$ are natural isometric embeddings and $\Psi':\mathcal M\to|\mathcal N|$ is a map induced by the projection $p$.
Then we have the following commutative diagram:

\begin{equation}\label{eq:Theta'}
\begin{CD}
M @>\tau>> \mathcal D \\
@V\Theta VV @VV\iota V \\
|\mathcal N| @<<\Psi'< \mathcal M.
\end{CD}
\end{equation}

The following claim is proved in exactly the same way as Claim \ref{clm:Psi}.

\begin{clm}\label{clm:Psi'}
$\Psi(|\mathcal N|)$ is a $(C,\epsilon)$-Lipschitz strong deformation retract of $\mathcal M$.
\end{clm}

\begin{rem}\label{rem:tau'}
On the other hand, $\tau(M)$ here is not necessarily a deformation retract of $\mathcal D$ (compare with Claim \ref{clm:tau}).
This is because $U_\sigma$ is no longer the intersection of $U_j$.
For this reason we have only the one-sided homotopy equivalence in Proposition \ref{prop:dom}.
\end{rem}

\step{Step 3}
The next claim is also proved in the same way as Claim \ref{clm:Phi}.

\begin{clm}\label{clm:Phi'}
There exists a $C$-Lipschitz retraction $\bar\Phi$ of $\mathcal M$ to $\iota(\mathcal D)$.
\end{clm}

\begin{rem}
In fact, there exists a $(C,\epsilon)$-Lipschitz strong deformation retraction $\Phi$ of $M$ to $\iota(\mathcal D)$ such that $\bar\Phi=\Phi(\cdot,\epsilon)$ as in Claim \ref{clm:Phi}.
However, we do not need this homotopy here.
\end{rem}

The proof is by reverse induction on skeleta.
The desired map $\bar\Phi$ is obtained as the compositions of $\bar\Phi_\sigma$ defined on each simplex $\sigma\in\mathcal N$.
Note that the condition (2) in Step 2 guarantees that the composition can be defined.
The construction of $\bar\Phi_\sigma=\Phi_\sigma(\cdot,\epsilon)$ is exactly the same as in Subclaim \ref{subclm:Phi}, so we omit the proof.

\step{Step 4}
The inverse map $\zeta:|\mathcal N|\to M$ is defined by the following commutative diagram:
\begin{equation}\label{eq:zeta'}
\begin{CD}
M @<q << \mathcal D \\
@A\zeta AA @AA\bar\Phi A \\
|\mathcal N| @>>\Psi> \mathcal M.
\end{CD}
\end{equation}
Now the statement follows from \eqref{eq:Theta'}, \eqref{eq:zeta'}, Claims \ref{clm:Psi'}, and \ref{clm:Phi'}.
\end{proof}

\subsection{Lipschitz homotopy}

By using dominations, we obtain

\begin{prop}\label{prop:hom}
Let $M$ and $M'$ be metric spaces that are locally $C$-Lipschitz contractible and locally $C$-doubling.
Suppose $C$-Lipschitz maps
\[f_i:M\to M',\quad i=0,\epsilon\]
are $\epsilon$-close in the uniform distance.
If $\epsilon$ is small enough, then there exists a $(\tilde C,\epsilon)$-Lipschitz homotopy between $f_0$ and $f_\epsilon$.
\end{prop}

\begin{proof}
Let $\Theta:M\to|\mathcal N|$ and $\zeta:|\mathcal N|\to M$ be $(C,\epsilon)$-Lipschitz dominations as in Proposition \ref{prop:dom}.
Then $f_i$ is $(C,\epsilon)$-Lipschitz homotopic to $f_i\circ\zeta\circ\Theta$.
Furthermore, Corollary \ref{cor:simp} together with the assumption implies that $f_i\circ\zeta$ ($i=0,\epsilon$) are $(C,\epsilon)$-Lipschitz homotopic.
Therefore,
\[f_0\sim f_0\circ\zeta\circ\Theta\sim f_\epsilon\circ\zeta\circ\Theta\sim f_\epsilon,\]
where $\sim$ denote $(C,\epsilon)$-Lipschitz homotopies.
\end{proof}

Now we are ready to prove Theorem \ref{thm:Peter}

\begin{proof}[Proof of Theorem \ref{thm:Peter}]
Let $M$ and $M'$ be as in the assumption that are $\epsilon$-close in the Gromov--Hausdorff distance, where $\epsilon$ is small enough.
Take a maximal $\epsilon/2$-discrete set $\{p_j\}_j$ of $M$ and construct a $C$-Lipschitz map
\[\Theta:M\to |\mathcal N|\]
as in the proof of Proposition \ref{prop:dom}.
Using the $\epsilon$-approximation, we choose a corresponding (not necessarily $\epsilon/2$-discrete) set $\{p_j'\}_j$ of $M'$.
By Proposition \ref{prop:simp}, we can construct a $C$-Lipschitz map
\[f:|\mathcal N|\to M'\]
such that $f(v_j)=p_j'$ for all $j$, where $v_j$ is the $j$-th vertex of $\mathcal N$ corresponding to $B(p_j,2\epsilon)$ (observe that $f$ is $C$-Lipschitz on $|\mathcal N^0|$).
Then
\[F:=f\circ\Theta\]
is a $C$-Lipschitz map from $M$ to $M'$.
Furthermore, since $F$ maps $p_j$ to a point $C\epsilon$-close to $p_j'$, it is $C\epsilon$-close to the original $\epsilon$-approximation.
Similarly, one can construct a $C$-Lipschitz map $F':M'\to M$ with the same property.
Since $F$ and $F'$ are based on the same approximation, we see that $F'\circ F$ and $F\circ F'$ are $C\epsilon$-close to $1_M$ and $1_{M'}$, respectively.
Then Proposition \ref{prop:hom} shows that these maps are $(C,C\epsilon)$-Lipschitz homotopic, which completes the proof.
\end{proof}

\section{Proof of Theorem \ref{thm:isom}}\label{sec:isom}

In this section we prove Theorem \ref{thm:isom}.
Although the proof is essentially the same as \cite[1.4]{MY:lip}, to clarify the Lipschitz constant estimate, we shall take a slightly different approach.

As in Subsection \ref{sec:mod}, we denote by $\mathcal R_M(\delta,\ell)$ the set of points in an Alexandrov space $M$ having $(n,\delta)$-strainers of length $>\ell$, where $\delta$ is less than a constant depending only on $n$.
We also denote by $\tau(\delta)$ a positive function depending only on $n$ such that $\tau(\delta)\to0$ as $\delta\to0$.
In what follows, we will abuse $\tau(\delta)$ to denote different functions.
We will also abuse $C$ to denote various positive constants, which depends only on $n$, $D$, and $v$, unless otherwise stated.

The following theorem was proved in \cite[9.8]{BGP} (cf.\ \cite{Shi}, \cite{WSS}, \cite{Y:conv}) by using the center of mass technique.

\begin{thm}\label{thm:reg}
There exists $C=C(n)>0$ satisfying the following.
Let $M,M'$ be $n$-dimensional Alexandrov spaces with Gromov--Hausdorff distance $\epsilon<\delta\ell$.
Then there exists a $\tau(\delta)$-almost isometric open embedding
\[g:\mathcal R_M(\delta,\ell)\to M'\]
that is $C\epsilon$-close to the original $\epsilon$-approximation on $\mathcal R_M(\delta,\ell)$.
\end{thm}

We will glue the above almost isometry with the Lipschitz homotopy approximation constructed in Theorem \ref{thm:main}.
The next proposition is essential in our gluing argument, which was also used in \cite{MY:lip}.
Here we will omit the subscript $M$ of $\mathcal R_M(\delta,\ell)$ to simplify the notation.

\begin{prop}\label{prop:diag}
There exists $C>0$ (independent of $n$, $\delta$, and $\ell$) satisfying the following.
Let $M$ be an $n$-dimensional Alexandrov space.
Let $\Delta$ denote the diagonal set in $M\times M$.
Then for any $0<\epsilon<\delta\ell$, $\mathcal R(\delta,\ell)^2\cap B(\Delta,\epsilon)$ admits a $(C,\epsilon)$-Lipschitz deformation retraction into $\Delta$ in $B(\Delta,\epsilon)$, i.e., there exists a $C$-Lipschitz map\[\mathbf\Psi:\mathcal R(\delta,\ell)^2\cap B(\Delta,\epsilon)\times[0,\epsilon]\to B(\Delta,\epsilon)\]
such that $\mathbf\Psi(\cdot,0)=1_{\mathcal R(\delta,\ell)^2\cap B(\Delta,\epsilon)}$, $\mathbf\Psi(\cdot,\epsilon)\in\Delta$, and $\mathbf\Psi(\cdot,t)|_\Delta=1_\Delta$ for any $0\le t\le\epsilon$.
\end{prop}

Note that we equip $M\times M$ with the $L_2$ distance to ensure that it is an Alexandrov space.
The following proof is slightly different from the previous one in \cite{MY:lip} to make the Lipschitz constant estimate clearer.

\begin{proof}
The deformation retraction $\mathbf\Psi$ is given by reparameterizing the gradient flow of the distance function $f=\dist_{S(\Delta,2\epsilon)}$ from $S(\Delta,2\epsilon)$.

\step{Step 1}
We first observe that $f$ is $(1-\tau(\delta))$-regular on $\mathcal R(\delta,\ell)^2\cap B(\Delta,\epsilon)\setminus\Delta$ in the direction to $\Delta$, i.e.,
\begin{equation}\label{eq:str}
df(\uparrow_{\mathbf x}^\Delta)>1-\tau(\delta)
\end{equation}
for any $\mathbf x\in\mathcal R(\delta,\ell)^2\cap B(\Delta,\epsilon)\setminus\Delta$.
As mentioned earlier, $\tau(\delta)$ below may represent different functions.

Let us write $\mathbf x=(x_1,x_2)$, where $x_i\in\mathcal R(\delta,\ell)$ and $0<|x_1x_2|<\sqrt 2\epsilon$.
Note that a point of $\Delta$ closest to $\mathbf x$ can be written as $\mathbf y=(y,y)$, where $y$ is a midpoint of a shortest path $x_1x_2$.
Assume first you can extend $x_1x_2$ beyond both endpoints to a shortest path $z_1z_2$ of length $2\sqrt 2\epsilon$ with the same midpoint $y$.
Then it is easy to see that $\mathbf z:=(z_1,z_2)$ is the unique closest point to $\mathbf x$ in $S(\Delta,2\epsilon)$ (the uniqueness follows from the fact that there are no branching geodesics in Alexandrov spaces).
In this case $f$ is $1$-regular at $\mathbf x$ in the direction to $\Delta$.
In general, by using strainers, one can find an ``almost geodesic extension'' $z_1x_1x_2z_2$ in the sense that
\[\tilde\angle yx_iz_i>\pi-\tau(\delta)\]
and $\mathbf z:=(z_1,z_2)$ lies in $S(\Delta,2\epsilon)$ (such extensions are almost unique in the sense that $\tilde\angle z_ix_iz_i'<\tau(\delta)$ for different $z_i'$).
Thus we obtain the $(1-\tau(\delta))$-regularity of $f$.
The details are left to the reader.

\step{Step 2}
Let $\mathbf\Phi$ be the gradient flow of $f$ defined on $\mathcal R(\delta,\ell)^2\cap B(\Delta,\epsilon)$.
For $\mathbf x\in\mathcal R(\delta,\ell)^2\cap B(\Delta,\epsilon)$, we denote by $T(\mathbf x)$ the minimum time for which $\mathbf\Phi(\mathbf x,t)$ reaches $\Delta$.
The above estimate \eqref{eq:str}, together with Lemma \ref{lem:reg}, shows $T(\mathbf x)<2\epsilon$.

Note that $\mathbf\Phi$ is uniformly $C$-Lipschitz for time $<2\epsilon$, where $C$ is a constant independent of any a priori constant.
This is a consequence of Proposition \ref{prop:lip} and the fact that $f$ is $(2\epsilon^{-1})$-concave on $B(\Delta,\epsilon)$ (the latter follows from triangle comparison; see \cite[2.4]{MY:loc}).

We show that the map
\[\mathbf x\mapsto T(\mathbf x)\]
is $2C$-Lipschitz.
Indeed, let $\mathbf y\in\mathcal R(\delta,\ell)^2\cap B(\Delta,\epsilon)$ and assume $T(\mathbf x)<T(\mathbf y)$.
Set $\mathbf x':=\mathbf\Phi(\mathbf x,T(\mathbf x))$ and $\mathbf y':=\mathbf\Phi(\mathbf y,T(\mathbf x))$.
Then we have $|\mathbf x'\mathbf y'|\le C|\mathbf x\mathbf y|$ and thus $|\Delta\mathbf y'|\le C|\mathbf x\mathbf y|$.
Since $\mathbf\Phi$ decreases the distance to $\Delta$ with velocity almost $1$ (\eqref{eq:str}, Lemma \ref{lem:reg}), we obtain $T(\mathbf y')\le2C|\mathbf x\mathbf y|$.
Therefore
\[T(\mathbf y)\le T(\mathbf x)+T(\mathbf y')\le T(\mathbf x)+2C|\mathbf x\mathbf y|.\]

Now we define $\mathbf\Psi:\mathcal R(\delta,\ell)^2\cap B(\Delta,\epsilon)\times[0,\epsilon]\to B(\Delta,\epsilon)$ by
\[\mathbf\Psi(\mathbf x,t):=\mathbf\Phi(\mathbf x,(t/\epsilon)T(\mathbf x)).\]
It is easy to check that $\mathbf\Psi$ satisfies the desired conclusion.
\end{proof}

Let us write $\mathbf\Psi=(\Psi_1,\Psi_2)\in M\times M$.
For $(x,y)\in\mathcal R(\delta,\ell)^2\cap B(\Delta,\epsilon)$, by joining the curves $\Psi_1(x,y,2t)$ on $[0,\epsilon/2]$ and $\Psi_2(x,y,\epsilon-2t)$ on $[\epsilon/2,\epsilon]$, we obtain

\begin{cor}\label{cor:diag}
There is a $C$-Lipschitz map $\Psi:\mathcal R(\delta,\ell)^2\cap B(\Delta,\epsilon)\times[0,\epsilon]\to M$, where $C$ is a constant independent of any a priori constant, such that
\[\Psi(x,y,0)=x,\quad\Psi(x,y,\epsilon)=y,\quad\Psi(x,x,t)=x\]
for any $(x,y)\in\mathcal R(\delta,\ell)^2\cap B(\Delta,\epsilon)$ and $0\le t\le\epsilon$.
\end{cor}

\begin{proof}[Proof of Theorem \ref{thm:isom}]
Assume that $M$ and $M'$ are $(C^{-1}\epsilon)$-close in the Gromov--Hausdorff distance, where $\epsilon<C^{-1}\delta\ell$.
By Theorem \ref{thm:main}, there exists a $(C,\epsilon)$-Lipschitz homotopy approximation $f:M\to M'$ with inverse $f':M'\to M$.
By Theorem \ref{thm:reg}, there exists a $\tau(\delta)$-almost isometric open embedding
\[g:\mathcal R_M(10\delta,\ell/10)\to M'.\]
Note that $f$ and $g$ are $\epsilon$-close on their common domain.
By the openness of $g$, we see that the image of $g$ contains $\mathcal R_{M'}(9\delta,\ell/9)$.

We first glue these two maps $f$ and $g$ by using Corollary \ref{cor:diag} (one can also use the center of mass technique as in \cite{MY:lip}).
Let
\[\Psi':\mathcal R_{M'}(10\delta,\ell/10)^2\cap B(\Delta,C\epsilon)\times[0,\epsilon]\to M'\]
be a $C$-Lipschitz map as in Corollary \ref{cor:diag} (we are rescaling the time interval as usual).
We define $h:M\to M'$ by
\[
h(x):=
\begin{cases}
\Psi'(g(x),f(x),\min\{|\mathcal R_M(2\delta,\ell/2),x|,\epsilon\}) & x\in\mathcal R_M(3\delta,\ell/3)\\
\hfil f(x) &  \text{otherwise}.
\end{cases}
\]
Note that $\epsilon$ is much smaller than $\delta$ and $\ell$.
Then $h$ is a $C$-Lipschitz $C\epsilon$-approximation such that $h=g$ on $\mathcal R_M(2\delta,\ell/2)$ and $h=f$ outside $\mathcal R_M(3\delta,\ell/3)$.
Similarly, we define $h':M'\to M$ by gluing $f'$ and $g^{-1}$ by using
\[\Psi:\mathcal R_{M}(10\delta,\ell/10)^2\cap B(\Delta,C\epsilon)\times[0,\epsilon]\to M\]
so that $h'=g^{-1}$ on $\mathcal R_{M'}(2\delta,\ell/2)$ and $h'=f'$ outside $\mathcal R_{M'}(3\delta,\ell/3)$.
Then we have
\[
h'\circ h(x)=
\begin{cases}
\hfil x & x\in\mathcal R_M(\delta,\ell) \\
f'\circ f(x) & x\notin\mathcal R_M(4\delta,\ell/4).
\end{cases}
\]

Let $F$ be a $(C,\epsilon)$-Lipschitz homotopy between $f'\circ f$ and $1_M$.
Set
\[G(x,t):=\Psi(h'\circ h(x),x,t),\]
which is a $(C,\epsilon)$-Lipschitz homotopy between $h'\circ h$ and $1_M$ defined on $\mathcal R_M(9\delta,\ell/9)$ and fixing $\mathcal R_M(\delta,\ell)$.
Since $h'\circ h=f'\circ f$ outside $\mathcal R_M(4\delta,\ell/4)$, one can glue these two homotopies $F$ and $G$ by using $\Psi$.
More precisely, we define $H:M\times[0,\epsilon]\to M$ by
\begin{align*}
&H(x,t)\\
&:=
\begin{cases}
\Psi(G(x,t),F(x,t),\min\{|\mathcal R_M(4\delta,\ell/4),x|,\epsilon\}) & x\in\mathcal R_M(5\delta,\ell/5)\\
\hfil F(x,t) & \text{otherwise}.
\end{cases}
\end{align*}
Then $H$ is a $(C,\epsilon)$-Lipschitz homotopy between $h'\circ h$ and $1_M$ that fixes $\mathcal R_M(\delta,\ell)$.
Similarly, one can construct a $(C,\epsilon)$-Lipschitz homotopy between $h\circ h'$ and $1_{M'}$ that fixes $\mathcal R_{M'}(\delta,\ell)$.
This completes the proof.
\end{proof}

\section{Proof of Theorem \ref{thm:ext}}\label{sec:ext}

In this section we give two proofs of Theorem \ref{thm:ext}.
The first is a direct corollary of the proofs of Theorem \ref{thm:main} and the second provides a more general way to deform a Lipschitz homotopy approximation to preserve extremal subsets.
Since we have proved Theorem \ref{thm:main} in two different ways, there are consequently three ways to prove Theorem \ref{thm:ext}.

We first define extremal subsets.
Here we state the definition in \cite{Pet:sc} in terms of gradient flows, which is more suitable for our applications than the original one in \cite{PP:ext}.

\begin{dfn}\label{dfn:ext}
A subset $E$ of an Alexandrov space $M$ is said to be \textit{extremal} if the gradient curve of any semiconcave function starting at $p\in E$ remains in $E$.
\end{dfn}

Note that any extremal subset is closed.
A typical example is the boundary of an Alexandrov space.
See \cite[\S4]{Pet:sc} or \cite{PP:ext} for other examples and properties of extremal subsets.

\subsection{First proof}

We prove that the Lipschitz homotopies constructed in Sections \ref{sec:main1} and \ref{sec:main2} automatically preserve extremal subsets (a slight modification is necessary for the latter).
For each statement in Sections \ref{sec:main1} and \ref{sec:main2}, we will provide a complement concerning extremal subsets.
The key fact is that the $(C,\epsilon)$-Lipschitz contraction of Theorem \ref{thm:conv} was given by gradient flows, and hence it automatically preserves extremal subsets.
Note that the contents of this subsection is irrelevant to Lipschitz continuity.

\begin{conv}
In the following complements, we always assume that $M$ and $M'$ are Alexandrov spaces in $\mathcal A(n,D,v)$ (not general metric spaces as in Sections \ref{sec:main1} and \ref{sec:main2}).
We also assume that $(C,\epsilon)$-Lipschitz contractible domains and $(C,\epsilon)$-good coverings are the ones obtained by the proofs of Theorems \ref{thm:conv} and \ref{thm:good} such that the contractions are defined by gradient flows.
\end{conv}

\subsubsection{Complements to Section \ref{sec:main1}}

Let $M\in\mathcal A(n,D,v)$ and let $\mathcal U=\{U_j\}_j$ be a $(C,\epsilon)$-good covering of $M$ given by Theorem \ref{thm:good}.
We use the same notation as in Section \ref{sec:main1}: $\mathcal N$ and $|\mathcal N|$ denote the nerve of $\mathcal U$ and its $\epsilon$-geometric realization, respectively.
Furthermore, for $\sigma\in\mathcal N$, we denote by $U_\sigma$ the corresponding intersection of $U_j$ and by $p_\sigma$ the center of the contraction of $U_\sigma$.

Let $E\subset M$ be an extremal subset.
We define a subcomplex $\mathcal N(E)$ of $\mathcal N$ as a subset consisting of $\sigma\in\mathcal N$ such that $U_\sigma$ intersects $E$.
We identify $|\mathcal N(E)|$ with a subset of $|\mathcal N|$.

Since the Lipschitz contraction of Theorem \ref{thm:good}(3) is given by gradient flows, we have

\begin{lem}\label{lem:cen}
If $\sigma\in\mathcal N(E)$, then $p_\sigma\in E$.
\end{lem}

Using this lemma, we prove

\begin{comp}[to Theorem \ref{thm:nerve}]\label{comp:nerve}
$(M,E)$ is $(\tilde C,\epsilon)$-Lipschitz homotopy equivalent to $(|\mathcal N|,|\mathcal N(E)|)$.
\end{comp}

\begin{proof}
Let $\Theta:M\to|\mathcal N|$ and $\zeta:|\mathcal N|\to M$ be the $(\tilde C,\epsilon)$-Lipschitz homotopy equivalences constructed in the proof of Theorem \ref{thm:nerve}.
We prove that $\Theta$ and $\zeta$ and their homotopies respect $E$ and $|\mathcal N(E)|$ (see Definition \ref{dfn:rel} for the terminology).

\step{Step 1}
We first show that $\Theta$ and $\zeta$ respect $E$ and $|\mathcal N(E)|$.
It is clear that $\Theta(E)\subset |\mathcal N(E)|$.
To show $\zeta(|\mathcal N(E)|)\subset E$, let us first recall how the map $\zeta$ is constructed.
Observe by \eqref{eq:zeta} that
\[\zeta=q\circ\bar\Phi\circ\Psi,\]
where $\Psi:|\mathcal N|\to\mathcal M$ is an embedding, $q:\mathcal D\to M$ is a projection, and $\bar\Phi=\Phi(\cdot,\epsilon):\mathcal M\to\iota(\mathcal D)=\mathcal D$ is a retraction that is a composition of $\bar\Phi_\sigma=\Phi_\sigma(\cdot,\epsilon)$ constructed in Subclaim \ref{subclm:Phi}.

In what follows, we will focus only on the ``$M$-coordinate'' in $\mathcal M$.
More precisely, if $y=(\theta,[x,t])\in\mathcal M\subset |\mathcal N|\times K(M)$ and if $t\neq\epsilon$, then we call $x$ the \textit{$M$-coordinate} of $y$.

Let $\theta\in|\mathcal N|$ and set $\sigma:=\supp\theta$.
By definition $\Psi(\theta)=(\theta,v_M)$, where $v_M$ is the vertex of $K(M)$.
Depending on $\theta$, there are two possibilities for its $\bar\Phi_\sigma$-image:
\begin{enumerate}
\item the $M$-coordinate of $\bar\Phi_\sigma(\theta,v_M)$ is $p_\sigma$; or
\item $\bar\Phi_\sigma(\theta,v_M)=(\eta,v_M)$ for some $\eta\in\partial\sigma$
\end{enumerate}
(see the proof of Subclaim \ref{subclm:Phi}; in fact there is no natural candidate for the $M$-coordinate other than $p_\sigma$ since $(\theta,v_M)$ has no $M$-coordinate).
Let us next consider $\bar\Phi_\tau\circ\bar\Phi_\sigma(\theta,v_M)$, where $\tau$ is a codimension $1$ face of $\sigma$ for which this composition makes sense.
In the case (1), we see that $p_\sigma$ is moved by the Lipschitz contraction of $U_\tau$ to the $M$-coordinate of $\bar\Phi_\tau\circ\bar\Phi_\sigma(\theta,v_M)$.
In the case (2), we can repeat the same argument as we did for $(\theta,v_M)$.

Repeating this observation, we see that in the $M$-coordinate, $\theta$ is first mapped to some $p_{\sigma'}$, where $\sigma'$ is a face of $\sigma$, and then moved by the Lipschitz contractions of $U_\tau$ to reach $\zeta(\theta)$, where $\tau$ are faces of $\sigma'$.
In particular, if $\theta\in|\mathcal N(E)|$ then $p_{\sigma'}\in E$ by Lemma \ref{lem:cen}.
Since $E$ is closed under gradient flows, this implies $\zeta(\theta)\in E$.

\step{Step 2}
Next we show that the homotopies for $\Theta$ and $\zeta$ respect $E$ and $|\mathcal N(E)|$.
We first consider the homotopy between $\zeta\circ\Theta$ and $1_M$, which is essentially given by Claim \ref{clm:Psi}.
More precisely, for $x\in M$, the homotopy of Claim \ref{clm:Psi} gives a curve from $x$ to $\zeta\circ\Theta(x)$ ($=q\circ\bar\Phi(\Theta(x),v_M)$) defined by
\[s\mapsto q\circ\bar\Phi(\Theta(x),[x,s])\]
for $s\in[0,\epsilon]$.
The same observation as in Step 1 shows that this is a broken gradient curve.
Therefore, if the starting point $x$ lies in $E$, it is entirely contained in $E$.

Next we consider the homotopy between $\Theta\circ\zeta$ and $1_{|\mathcal N|}$, which is essentially given by Claims \ref{clm:tau} and \ref{clm:Phi}.
Recall that $p:\mathcal D\to|\mathcal N|$ is a projection.
For $\theta\in|\mathcal N|$, the homotopy of Claim \ref{clm:tau} gives a line segment between $p\circ\bar\Phi(\theta,v_M)$ and $\Theta\circ\zeta(\theta)$ (=$\Theta(q\circ\bar\Phi(\theta,v_M))$) defined by
\[s\mapsto p\circ H(\bar\Phi(\theta,v_M),s)\]
for $s\in[0,\epsilon]$.
Observe that this line segment lies in $\supp\Theta\circ\zeta(\theta)$.
Furthermore, the homotopy of Claim \ref{clm:Phi} gives a broken line joining $\theta$ and $p\circ\bar\Phi(\theta,v_M)$ defined by
\[s\mapsto p\circ\Phi((\theta,v_M),s)\]
for $s\in[0,\epsilon]$.
This broken line lies in $\supp\theta$.
The homotopy between $\theta$ and $\Theta\circ\zeta(\theta)$ is the concatenation of the above two homotopies.
If $\theta\in|\mathcal N(E)|$, both $\supp\Theta\circ\zeta(\theta)$ and $\supp\theta$ are contained in $|\mathcal N(E)|$ (the former follows from Step 1).
Therefore the homotopy is entirely contained in $|\mathcal N(E)|$.
\end{proof}

To prove Theorem \ref{thm:ext}, we also need the following stability.

\begin{comp}[to Theorem \ref{thm:good}(4)]\label{comp:good}
Let $E$ and $E'$ be extremal subsets of $M$ and $M'$, respectively, that are $(C^{-1}\epsilon)$-close via the Gromov--Hausdorff approximation between $M$ and $M'$.
Then the subcomplexes of the nerves of $\mathcal U$ and $\mathcal U'$ defined by $E$ and $E'$ are the same.
\end{comp}

\begin{proof}
Let $U$ be an intersection of elements of $\mathcal U$ and let $U'$ denote its lift to $M'$.
We show that $U$ intersects $E$ if and only if  $U'$ intersects $E'$.
Let $p$ and $p'$ denote the peaks of $U$ and $U'$, i.e., the unique maximum points of the strictly concave functions $h$ and $h'$ defining $U$ and $U'$, respectively.
By Claim \ref{clm:good}, $U$ and $U'$ contain the $(c_0\epsilon)$-balls around $p$ and $p'$, respectively.
We may assume $C^{-1}\ll c_0$.

We first show that $p$ is $(c_0\epsilon/3)$-close to $p'$ via the Gromov--Hausdorff approximation.
Suppose this does not hold.
Since $h'$ is a lift of $h$, the maximum value $h'(p')$ is $(C^{-1}\epsilon)$-close to $h(p)$.
If $p$ is not $(c_0\epsilon/3)$-close to $p'$, then by lifting it we find
\[p''\in M'\setminus B(p',c_0\epsilon/4)\]
such that $h'(p'')$ is also $(C^{-1}\epsilon)$-close to $h(p)$.
Using the gradient estimate (Lemma \ref{lem:grad}), one can increase the value of $h'$ by moving $p''$ through the gradient flow of $h'$.
Since $C^{-1}\ll c_0$, this gives a point of $M'$ with $h'$-value greater than $h'(p')$.
This is a contradiction.

Suppose $U$ intersects $E$.
It also follows from the gradient estimate and the extremality of $E$ that $p\in E$ (indeed, otherwise one finds a point of $E\setminus\{p\}$ attaining the maximum value of $h|_E$, but then the gradient flow of $h$ going through $E$ increases $h|_E$).
Since $E'$ is $(C^{-1}\epsilon)$-close to $E$ and $p'$ is $(c_0\epsilon/3)$-close to $p\in E$, where $C^{-1}\ll c_0$, we find a point of $E'$ that is $(c_0\epsilon/2)$-close to $p'$.
Since $U'$ contains the $(c_0\epsilon)$-ball around $p'$, $U'$ intersects $E'$.
The converse is proved in exactly the same way.
\end{proof}

Theorem \ref{thm:ext} immediately follows from Complements \ref{comp:nerve} and \ref{comp:good} in the same way as Theorem \ref{thm:main} in Section \ref{sec:main1}.

\begin{rem}
It is worth mentioning that $\mathcal N(E)$ is a \textit{full subcomplex} of $\mathcal N$, that is, a simplex $\sigma\in\mathcal N$ belongs to $\mathcal N(E)$ if and only if every vertex of $\sigma$ belongs to $\mathcal N(E)$.
For the proof, see \cite[5.14]{F:good} and the subsequent paragraph.
It is known that any full subcomplex has a regular neighborhood admitting a deformation retraction onto it (\cite[Ch.\ II \S9]{EF}).
Compare with Theorem \ref{thm:ndr}.
\end{rem}

\subsubsection{Complements to Section \ref{sec:main2}}

Let $M\in\mathcal A(n,D,v)$ and $E$ an extremal subset.
By the proofs and the fact that $(C,\epsilon)$-Lipschitz contractions preserve extremal subsets, we have

\begin{comp}[to Proposition \ref{prop:simp}]\label{comp:simp}
Let $\sigma\in K$ be a simplex with
\[f(v)\in E,\quad f(\sigma\cap|L|)\subset E\]
for any vertex $v\in\sigma$.
Then the extension $\tilde f$ satisfies $\tilde f(\sigma)\subset E$.
\end{comp}

\begin{comp}[to Corollary \ref{cor:simp}]\label{comp:simp'}
Let $L$ be a subcomplex of $K$ with
\[f_i(|L|)\subset E\]
for $i=0,\epsilon$.
Then the homotopy between $f_0$ and $f_\epsilon$ respects $|L|$ and $E$, i.e., the images of $|L|$ are always contained in $E$.
\end{comp}

As in Subsection \ref{sec:dom}, let $\mathcal N$ be the nerve of a covering of $M$ consisting of $2\epsilon$-balls centered at maximal $\epsilon/2$-discrete points $\{p_j\}_j$ and $|\mathcal N|$ its $\epsilon$-geometric realization.
As before, we define a subcomplex $\mathcal N(E)$ of $\mathcal N$ as a subset consisting of simplices such that the corresponding intersections of $2\epsilon$-balls intersect $E$.

The following is proved in the same way as Complement \ref{comp:nerve}.

\begin{comp}[to Proposition \ref{prop:dom}]\label{comp:dom}
$(M,E)$ is $(\tilde C,\epsilon)$-Lipschitz dominated by $(|\mathcal N|,|\mathcal N(E)|)$, i.e., the maps $\Theta$ and $\zeta$, and the homotopy between $\zeta\circ\Theta$ and $1_M$ respect $E$ and $|\mathcal N(E)|$.
\end{comp}

By Complements \ref{comp:simp'} and \ref{comp:dom} we have

\begin{comp}[to Proposition \ref{prop:hom}]\label{comp:hom}
Let $E$ and $E'$ be extremal subsets in $M$ and $M'$, respectively, such that
\[f_i(E)\subset E'\]
for $i=0,\epsilon$.
Then the homotopy between $f_0$ and $f_\epsilon$ respects $E$ and $E'$.
\end{comp}

Theorem \ref{thm:ext} follows from Complements \ref{comp:simp}, \ref{comp:dom} and \ref{comp:hom} in almost the same way as Theorem \ref{thm:main} in Section \ref{sec:main2}.
The only modification is as follows.
Let $M,M'\in\mathcal A(n,D,v)$ and $E\subset M$, $E'\subset M'$ extremal subsets that are sufficiently close in the Gromov--Hausdorff distance.
When constructing a map
\[F=f\circ\Theta:M\to M'\]
as in the proof of Theorem \ref{thm:Peter}, we now have to choose $\{p_j'\}_j\subset M'$ so that $p_j'\in E'$ whenever $B(p_j,2\epsilon)$ intersects $E$.
Then Complements \ref{comp:simp} and \ref{comp:dom} imply that $F(E)\subset E'$.
Similarly we get a map $F':M'\to M$ satisfying $F'(E')\subset E$.
The rest follows from Complement \ref{comp:hom}.

\subsection{Second proof}

We prove that any Lipschitz homotopy approximation between Alexandrov spaces can be deformed to preserve nearby extremal subsets.

The proof is based on the following theorem.
Fix $n$, $D$, and $v$.

\begin{thm}[{\cite[1.5]{F:reg}}]\label{thm:ndr}
There is $C=C(n,D,v)>0$ satisfying the following:
Let $M\in\mathcal A(n,D,v)$ and $E\subset M$ an extremal subset.
Then for any $0<\epsilon<C^{-1}$, the $\epsilon$-neighborhood $B(E,\epsilon)$ of $E$ admits a $(C,\epsilon)$-Lipschitz strong deformation retraction to $E$.
\end{thm}

Although the original statement of \cite[1.5]{F:reg} does not assert the existence of such a uniform constant, one can easily check it by combining the original proof with the following lemma.

\begin{lem}\label{lem:ndr}
Let $M\in\mathcal A(n,D,v)$ and $E\subset M$ a proper extremal subset.
Then there exists $q\in M$ such that
\[|qE|>c(n,D,v),\]
where $c(n,D,v)$ is a positive constant depending only on $n$, $D$, and $v$.
\end{lem}

\begin{proof}
Any limit of proper extremal subsets under a noncollapsing convergence of Alexandrov spaces is a proper extremal subset (\cite[9.13]{K:stab}).
Hence the statement follows from the compactness of $\mathcal A(n,D,v)$.
\end{proof}

\begin{proof}[Proof of Theorem \ref{thm:ndr}]
We explain how to modify the original proof.
It suffices to check that every constant appearing in the original proof can be chosen depending only on $n$, $D$, and $v$.
In what follows, we denote respectively by $C$ and $c$ various large and small positive constants depending only on $n$, $D$, and $v$.

Recall that the Lipschitz deformation retraction of \cite[1.5]{F:reg} was obtained by reparameterizing the gradient flow $\Phi$ of a function
\[h=\frac1N\sum_{\alpha=1}^N|q_\alpha\cdot|.\]
Here $\{q_\alpha\}_{\alpha=1}^N$ is a $\nu$-discrete set contained in $M\setminus E$ such that $N\ge c/\nu^n$ for sufficiently small $\nu>0$, where $c=c(n,D,v)$.
By Lemma \ref{lem:ndr}, $\{q_\alpha\}_{\alpha=1}^N$ can be chosen so that $\rho:=\min_\alpha|q_\alpha E|>c$.
In particular, in the $\rho/2$-neighborhood of $E$, the concavity of $h$ depends only on $\rho$, which in turn depends only on $n$, $D$, and $v$.
Hence the gradient flow $\Phi$ of $h$ is uniformly $C$-Lipschitz in the $\rho/2$-neighborhood of $E$, where $C=C(n,D,v)$ (as long as the time is uniformly bounded above; see Proposition \ref{prop:lip}).

In \cite[A.1]{F:reg}, we showed that
\[d\dist_E(\nabla h)<-c\]
on $B(E,\epsilon)\setminus E$, where $\epsilon\ll\rho$.
Since the proof relies only on volume comparison, one can check that the constants $c$ and $\epsilon$ depend only on $n$, $D$, and $v$.

The rest of the proof is the same as Step 2 of the proof of Proposition \ref{prop:diag}.
For $x\in B(E,\epsilon)$, let $T(x)$ denote the minimum time such that the gradient flow $\Phi$ starting at $x$ reaches $E$.
From \cite[A.1]{F:reg} above, $T(x)$ is uniformly bounded above by $C\epsilon$, where $C=C(n,D,v)$.

In \cite[A.2]{F:reg}, we showed that the function
\[x\mapsto T(x)\]
is Lipschitz.
Since the proof relies only on \cite[A.1]{F:reg} and the Lipschitz continuity of the gradient flow $\Phi$, one can check that it is uniformly $C$-Lipschitz, where $C=C(n,D,v)$.

The desired deformation retraction $\Psi:B(E,\epsilon)\times[0,\epsilon]\to B(E,\epsilon)$ is given by
\[\Psi(x,t):=\Phi\left(x,(t/\epsilon)T(x)\right).\]
A direct calculation shows that $\Psi$ is also uniformly $C$-Lipschitz.
\end{proof}

\begin{rem}\label{rem:ndr}
As can be seen in the proof, the flow $\Psi$ monotonically decreases the distance to $E$.
This will be used below.
\end{rem}

We now give yet another proof of Theorem \ref{thm:ext}, which actually provides a way to deform a given Lipschitz homotopy approximation to preserve extremal subsets.

\begin{proof}[Proof of Theorem \ref{thm:ext}]
As above, we denote by $C$ various large positive constants depending only on $n$, $D$, and $v$.
Let $M,M'\in\mathcal A(n,D,v)$ and $E\subset M$, $E'\subset M'$ be extremal subsets that are $(C^{-1}\epsilon)$-close in the Gromov--Hausdorff distance, where $0<\epsilon<C^{-1}$.
By Theorem \ref{thm:main}, we may assume there exist $(C,\epsilon)$-Lipschitz homotopy approximations
\[f:M\to M',\quad f':M'\to M\]
with $(C,\epsilon)$-Lipschitz homotopies
\[F:M\times[0,\epsilon]\to M,\quad F':M'\times[0,\epsilon]\to M'\]
connecting $f'\circ f$ and $f\circ f'$ to the identities $1_M$ and $1_{M'}$, respectively (however, these maps do not necessarily have to be the ones given in the proof of Theorem \ref{thm:main} that automatically preserve extremal subsets).
If necessary, taking the Gromov--Hausdorff distance smaller by multiplying by a constant, we may assume that
\begin{gather*}
f(E)\subset B(E',\epsilon),\quad f'(E')\subset B(E,\epsilon),\\
F(E,t)\subset B(E,\epsilon),\quad F'(E',t)\subset B(E',\epsilon)
\end{gather*}
for all $0\le t\le\epsilon$.
We deform these maps to respect $E$ and $E'$.

Let $\Psi$ and $\Psi'$ be, respectively, the $(C,\epsilon)$-Lipschitz strong deformation retractions of $B(E,10\epsilon)$ and $B(E',10\epsilon)$ to $E$ and $E'$ provided by Theorem \ref{thm:ndr}.
We extend $\Psi$ to $\tilde\Psi:M\times[0,\epsilon]\to M$ as follows:
\[
\tilde\Psi(x,t):=
\begin{cases}
\Psi(x,\rho(|Ex|)t) & x\in B(E,10\epsilon)\\
\hfil x & \text{otherwise},
\end{cases}
\]
where $\rho:[0,\infty)\to[0,1]$ is a $(5\epsilon)^{-1}$-Lipschitz function such that $\rho\equiv1$ on $[0,5\epsilon]$ and $\rho\equiv0$ on $[10\epsilon,\infty)$.
It is straightforward to check that $\tilde \Psi$ is $C$-Lipschitz and $\tilde\Psi(\cdot,0)=1_M$, where $C$ is independent of $\epsilon$.
We also define $\tilde\Psi':M'\times[0,\epsilon]\to M'$ in the same manner.
Let $\tilde\Psi_t$ and $\tilde\Psi_t'$ denote $\tilde\Psi(\cdot,t)$ and $\tilde\Psi'(\cdot,t)$, respectively.

Now we define new maps $\tilde f:M\to M'$ and $\tilde f':M'\to M$ by
\[\tilde f:=\tilde \Psi_\epsilon'\circ f,\quad\tilde  f':=\tilde\Psi_\epsilon\circ f'.\]
Then they are $C$-Lipschitz $10\epsilon$-approximations satisfying $\tilde f(E)\subset E'$ and $\tilde f'(E')\subset E$.
Moreover,
\[\tilde f'\circ\tilde f=\tilde\Psi_\epsilon\circ f'\circ\tilde \Psi'_\epsilon\circ f\sim\tilde\Psi_\epsilon\circ f'\circ f\sim\tilde\Psi_\epsilon\sim1_M,\]
where $\sim$ denote $(C,\epsilon)$-Lipschitz homotopies.
It is easy to see that these homotopies respect $E$ (note $\tilde\Psi_t'\circ f(E)\subset B(E',\epsilon)$ by Remark \ref{rem:ndr}).
Similarly, $\tilde f\circ\tilde f'$ is $(C,\epsilon)$-Lipschitz homotopic to $1_{M'}$ while respecting $E'$.
This completes the proof.
\end{proof}

\end{document}